\documentclass[11pt,draft,dvips]{amsart}  
\usepackage{amssymb,amsmath,amsfonts}  
\usepackage[dvips]{graphicx,color}
\usepackage{cite}
\allowdisplaybreaks[1]

\definecolor{pgray}{gray}{0.8}

\newtheorem{theorem}{Theorem}[section]

\newtheorem{proposition}[theorem]{Proposition}
\newtheorem{lemma}[theorem]{Lemma}

\numberwithin{equation}{section}

\title{\mbox{}}

\begin{document}
\begin{center}
{\bf \Large{
	Asymptotic Expansion of Solutions to the Drift-Diffusion Equation with Fractional Dissipation
}}\\
\vspace{5mm}
{\sc\large
	Masakazu Yamamoto}\footnote{Graduate School of Science and Technology,
	Niigata University,
	Niigata 950-2181, Japan}\qquad
{\sc\large
	Yuusuke Sugiyama}\footnote{Department of Mathematics,
Tokyo University of Science,
Tokyo 162-8601, Japan}
\end{center}
\maketitle
\vspace{-15mm}
%
\begin{abstract}
The initial-value problem for the drift-diffusion equation arising from the model of semiconductor device simulations is studied.
The dissipation on this equation is given by the fractional Laplacian $(-\Delta)^{\theta/2}$.
Large-time behavior of solutions to the drift-diffusion equation with $0 < \theta \le 1$ is discussed.
When $\theta > 1$, large-time behavior of solutions is known.
However, when $0 < \theta \le 1$, the perturbation methods used in the preceding works would not work.
Large-time behavior of solutions to the drift-diffusion equation with $0 < \theta \le 1$ is discussed.
Particularly, the asymptotic expansion of solutions with high-order is derived.
\end{abstract}

\section{Introduction}
We study the following initial-value problem for the drift-diffusion model for semiconductors:
\begin{equation}\label{DD}
\left\{
\begin{array}{lr}
	\partial_t u + (-\Delta)^{\theta/2} u - \nabla\cdot(u\nabla\psi) = 0,
	&
	t>0,~ x \in \mathbb{R}^n,\\
	-\Delta\psi = u,
	&
	t>0,~ x \in \mathbb{R}^n,\\
	u(0,x) = u_0 (x),
	&
	x \in \mathbb{R}^n,
\end{array}
\right.
\end{equation}
where $n \ge 2,~ 0 < \theta < n,~ \partial_t = \partial / \partial t,~ \nabla = (\partial_1,\ldots,\partial_n),~ \partial_j = \partial / \partial x_j,~ \Delta=\partial_1^2 + \cdots + \partial_n^2$, and $(-\Delta)^{\theta/2} \varphi = \mathcal{F}^{-1} [|\xi|^\theta \mathcal{F} [\varphi]]$.
The unknown functions $u$ and $\psi : (0,\infty) \times \mathbb{R}^n \to \mathbb{R}$ stand for the density of electrons and the potential of electromagnetic field, respectively.
The drift-diffusion equation with $\theta = 2$ is derived from conservation of mass of electrons.
The fractional Laplacian is associated to the jumping process in the stochastic process.
Since electrons on a semiconductor may jump from a dopant to another, the fractional Laplacian is suitable to describe their dissipation.
In the case $\theta > 1$, well-posedness and global existence of solutions are shown.
Moreover, large-time behavior of the solution is discussed (cf. \cite{Bl-Dl,K-K,K-O,M-T,M,O-S,Ym}).
When $\theta > 1$, we can refer to many preceding works to derive the asymptotic expansion of the solution of \eqref{DD} as $t\to\infty$ (cf. for example \cite{Cr,E-Z,F-M,I-I-K,I-K,I-K-K,KtM_B,N-S-U}).
In this case, perturbation methods are effective, since the highest-order derivative is on the dissipation term.
When $\theta = 1,~ \nabla u$ on the nonlinear term balances the dissipation $(-\Delta)^{1/2} u$.
In the case $0 < \theta < 1$, the highest-order derivative is on the nonlinear term.
Therefore, the perturbation methods would not work as discussed in more detail later.
In \cite{Y-S}, employing the energy method, the authors estimate the difference between the solution of \eqref{DD} with $\theta = 1$ and its second-order asymptotic expansion in $L^q (\mathbb{R}^n)$ for $1 < q < \infty$.
But the cases $q = 1$ and $q = \infty$ are excepted.
The purpose of this paper is to give the third-order asymptotic expansion for \eqref{DD} with $0 < \theta \le 1$.
Especially, we will estimate the difference between the solution and the asymptotic expansion in $L^q (\mathbb{R}^n)$ with $1 \le q \le \infty$.
Our main theorems are extensions from the results of the case $1 < \theta \le 2$ in \cite{Ym} to $0 < \theta \le 1$.
For the drift-diffusion equation with $0 < \theta \le 1$, we refer to the preceding works for the following two-dimensional quasi-geostrophic equation:
\[
	\left\{
	\begin{split}
		& \partial_t u + (-\Delta)^{\theta/2} u - \nabla^{\bot} \psi \cdot \nabla u = 0,
		& t > 0,~ x \in \mathbb{R}^2,\\
		& (-\Delta)^{1/2} \psi = u,
		& t > 0,~ x \in \mathbb{R}^2,
	\end{split}
	\right.
\]
where $\nabla^{\bot} = (- \partial_2, \partial_1)$.
For the quasi-geostrophic equation, well-posedness, and global in time existence for small initial data in the scale-invariant Besov spaces is shown (cf.\cite{Ch-L,C-C-W,Crdb-Crdb}).
By analogy from the method in the quasi-geostrophic equation, for the drift-diffusion equation with $0 < \theta \le 1$, well-posedness and global existence for small initial data in the scale-invariant Besov space are shown in \cite{S-Y-K}.
In \cite{S-Y-K}, global existence for positive initial data is also studied.
We consider the solution such that
\begin{equation}\label{7}
	u \in L^\infty \bigl( 0,\infty; L^1 (\mathbb{R}^n) \cap L^\infty (\mathbb{R}^n) \bigr),
	\quad
	\left\| u (t) \right\|_{L^p (\mathbb{R}^n)}
	\le
	C (1+t)^{-\frac{n}\theta (1-\frac1p)}
\end{equation}
for $1 \le p \le \infty$, and
\begin{equation}\label{8}
	u \in C^\infty \left( (0,\infty),H^\infty (\mathbb{R}^n) \right).
\end{equation}
In \cite{Br-Cl,L-L-Z,S-Y-K,Y-S-K}, it is shown that solutions satisfy \eqref{7} and \eqref{8}, if initial data are sufficiently smooth and nonnegative.
Upon the above assumption, the conservative force fulfills that
\begin{equation}\label{decay-pot}
	\bigl\| \nabla \psi (t) \bigr\|_{L^p (\mathbb{R}^n)}
	\le
	C (1+t)^{-\frac{n}\theta (1-\frac1p) + \frac1\theta}
\end{equation}
for $\frac{n}{n-1} < p \le \infty$ (see Proposition \ref{HLSinfty} in Section \ref{2}).
To discuss large-time behavior of the solution, we introduce the fundamental solution of $\partial_t u + (-\Delta)^{\theta/2} u = 0$:
\[
	G_\theta (t,x) = \mathcal{F}_\xi^{-1} \bigl[ e^{-t |\xi|^\theta} \bigr] (x).
\]
In the case $\theta = 1$, this function equals to the Poisson kernel
\[
	P (t,x)
	=
	\pi^{-\tfrac{n+1}2}\Gamma(\tfrac{n+1}2)
	\, t\left( t^2 + |x|^2 \right)^{-\tfrac{n+1}2}.
\]
The Duhamel formulae rewrites the solution of \eqref{DD} by the mild solution as follows:
\begin{equation}\label{MS}
	u(t) = G_\theta (t) * u_0 + \int_0^{t/2} \nabla G_\theta (t-s) * (u\nabla (-\Delta)^{-1}u) (s) ds + \int_{t/2}^t G_\theta (t-s) * \nabla\cdot (u\nabla (-\Delta)^{-1}u) (s) ds.
\end{equation}
We remark that, in the case $\theta > 1$, the second and the third terms are combined into $\int_0^t \nabla G_\theta (t-s) * (u\nabla (-\Delta)^{-1}u) (s) ds \in C ([0,T],L^1 (\mathbb{R}^n) \cap L^\infty (\mathbb{R}^n))$ since $\nabla G_\theta \in L^1 (0,T,L^1 (\mathbb{R}^n))$ and $u\nabla(-\Delta)^{-1}u \in L^\infty (0,T,L^1 (\mathbb{R}^n) \cap L^\infty (\mathbb{R}^n))$.
But, if $\theta \le 1$, $\nabla G_\theta \not\in L^1 (0,T,L^1 (\mathbb{R}^n))$, which requires estimates for $\nabla u$.
Furthermore, the third-order asymptotic expansion needs some estimates for $xu$ (see the remark after Theorem \ref{th1}).
However, \eqref{MS} does not work in those estimates, since the third term of \eqref{MS} contains $\nabla u$.
Employing the energy method with Kato and Ponce's commutator estimate and the positivity lemma for the fractional Laplacian, we get those estimate for $\nabla u$ and $xu$ respectively (see Propositions \ref{lem-dr} and \ref{lem-wt}).
Our first assertion is established as follows.
\begin{theorem}\label{th1}
	Let $n = 3$ and $0 < \theta < 1$, or $n \ge 4$ and $0 < \theta \le 1$.
	Assume that $u_0 \in L^1 (\mathbb{R}^n,(1+|x|^2)dx) \cap L^\infty (\mathbb{R}^n)$ and the solution $u$ satisfies \eqref{7} and \eqref{8}.
	Then
	\[
	\begin{split}
		\biggl\| &u(t) - M G_\theta (t) - m \cdot \nabla G_\theta (t)
		- \sum_{|\alpha| = 2} \frac{\nabla^\alpha G_\theta (t)}{\alpha!} \int_{\mathbb{R}^n} (-y)^\alpha u_0 (y) dy\\
		&- \sum_{|\beta| = 1} \nabla^\beta \nabla G_\theta (t) \cdot \int_0^\infty \int_{\mathbb{R}^n} (-y)^\beta \left( u\nabla(-\Delta)^{-1} u \right) (s,y)  dyds
		\biggr\|_{L^q (\mathbb{R}^n)}
		=
		o \bigl( t^{-\frac{n}\theta (1-\frac1q) - \frac2\theta} \bigr)
	\end{split}
	\]
	as $t \to \infty$ for $1 \le q \le \infty$, where $M = \int_{\mathbb{R}^n} u_0 (y) dy$ and $m = \int_{\mathbb{R}^n} (-y) u_0 (y) dy$.
\end{theorem}
We remark that the decay properties of $u$ ensure that
$
	\int_{\mathbb{R}^n} |y| |u\nabla(-\Delta)^{-1}u| dy
	 \le C (1+s)^{-\frac{n-2}\theta}
$
(see Proposition \ref{lem-wt}).
Hence the coefficient $\int_0^\infty \int_{\mathbb{R}^n} (-y)^\beta (u\nabla(-\Delta)^{-1}u) (s,y) dy ds$ in Theorem \ref{th1} converges to a finite value since $\theta < n-2$.
However, if $\theta \ge n-2$, this coefficient may diverge to infinity.
In this case, we should include some correction terms in the asymptotic expansion.
When $n = 2$ and $0 < \theta \le 1$, let $J$ be given by
\begin{equation}\label{J}
	J(t) = \int_0^{t/2} \nabla G_\theta (t-s) * (G_\theta \nabla (-\Delta)^{-1} G_\theta) (s) ds + \int_{t/2}^t G_\theta (t-s) * \nabla \cdot (G_\theta \nabla (-\Delta)^{-1} G_\theta) (s) ds.
\end{equation}
Then the same argument as in \cite{Y} yields that
\[
	J \in C \left( (0,\infty), L^1 (\mathbb{R}^2) \cap L^\infty (\mathbb{R}^2) \right),\quad
	J \neq 0,
\]
and
\begin{equation}\label{scJ}
	\| J(t) \|_{L^q (\mathbb{R}^2)} = t^{-\frac2\theta(1-\frac1q)-\frac{2-\theta}\theta} \| J(1) \|_{L^q (\mathbb{R}^2)}
\end{equation}
for $1 \le q \le \infty$.
Moreover, $J$ satisfies the following.
\begin{theorem}\label{th2}
	Let $n = 2,~ 0 < \theta < 1,~ u_0 \in L^1 (\mathbb{R}^2,(1+|x|^2)dx) \cap L^\infty (\mathbb{R}^2)$, and $J$ be given by \eqref{J}.
	Assume that the solution $u$ satisfies \eqref{7} and \eqref{8}.
	Then
	\[
	\begin{split}
		\biggl\| &u(t) - M G_\theta (t) - m \cdot \nabla G_\theta (t) - M^2 J(t)
		- \sum_{|\alpha| = 2} \frac{\nabla^\alpha G_\theta (t)}{\alpha!} \int_{\mathbb{R}^2} (-y)^\alpha u_0 (y) dy\\
		&- \sum_{|\beta| = 1} \nabla^\beta \nabla G_\theta (t) \cdot \int_0^\infty \int_{\mathbb{R}^2} (-y)^\beta \left( u\nabla(-\Delta)^{-1} u - M^2 G_\theta \nabla (-\Delta)^{-1} G_\theta  \right)(s,y) dyds
		\biggr\|_{L^q (\mathbb{R}^2)}\\
		=&
		o \bigl( t^{-\frac{2}\theta (1-\frac1q) - \frac2\theta} \bigr)
	\end{split}
	\]
	as $t \to \infty$ for $1 \le q \le \infty$, where $M = \int_{\mathbb{R}^2} u_0 (y) dy$ and $m = \int_{\mathbb{R}^2} (-y) u_0 (y) dy$.
\end{theorem}
Before the proof of this theorem (see the remark under the proof of Proposition \ref{HLSth2} in Section \ref{prf}), we will confirm that
\begin{equation}\label{coef2}
	\int_0^\infty \int_{\mathbb{R}^2} |y| |u\nabla(-\Delta)^{-1} u - M^2 G_\theta \nabla (-\Delta)^{-1} G_\theta| dyds < +\infty.
\end{equation}
Unfortunately, when $n = 3$ and $\theta = 1$, the first term of $J$ may diverge to infinity since $P\nabla(-\Delta)^{-1}P (s)$ is too singular as $s \to 0$.
For this case, we define
\begin{equation}\label{tJ}
\begin{split}
	\tilde{J} (t)
	=&
	\int_0^{t/2} \int_{\mathbb{R}^3}
		\left( \nabla P (t-s,x-y) + (y\cdot\nabla) \nabla P (t,x) \right)
		\cdot \left( P \nabla (-\Delta)^{-1} P \right) (s,y)
	dyds\\
	&+
	\int_{t/2}^t
		P (t-s) * \nabla \cdot \left( P \nabla (-\Delta)^{-1} P \right) (s)
	ds,\\
	\tilde{K} (t)
	=&
	\frac13 \Delta P (t) \log (1+\tfrac{t}2) \int_{\mathbb{R}^3} (-y) \cdot (P \nabla (-\Delta)^{-1} P) (1,y) dy.
\end{split}
\end{equation}
The function $\tilde{J}$ fulfills
\begin{equation}\label{tJp}
	\tilde{J} \in C \left( (0,\infty), L^1 (\mathbb{R}^3) \cap L^\infty (\mathbb{R}^3) \right),
	\quad
	\bigl\| \tilde{J} (t) \bigr\|_{L^q (\mathbb{R}^3)}
	=
	t^{-3 (1-\frac1q) - 2} \bigl\| \tilde{J} (1) \bigr\|_{L^q (\mathbb{R}^3)}
\end{equation}
for $1 \le q \le \infty$ (see Proposition \ref{proptJ} in Section 3), and provides the asymptotic expansion for the solution as follows.
\begin{theorem}\label{th3}
	Let $n = 3,~ \theta = 1,~ u_0 \in L^1 (\mathbb{R}^3, (1+|x|^2) dx) \cap L^\infty (\mathbb{R}^3)$, and $\tilde{J}$ and $\tilde{K}$ be given by \eqref{tJ}.
	Assume that the solution $u$ satisfies \eqref{7} and \eqref{8}.
	Then
	\[
	\begin{split}
		\biggl\|
			&u(t) - M P (t) - m \cdot \nabla P (t) - M^2 \tilde{K}(t) - M^2 \tilde{J} (t)
			- \sum_{|\alpha| = 2} \frac{\nabla^\alpha P (t)}{\alpha!} \int_{\mathbb{R}^3} (-y)^\alpha u_0 (y) dy\\
			&-\sum_{|\beta| = 1} \nabla^\beta \nabla P (t) \cdot \int_0^\infty \int_{\mathbb{R}^3}
				(-y)^\beta \left( u \nabla (-\Delta)^{-1} u(s,y) - M^2 P \nabla (-\Delta)^{-1} P (1+s,y) \right)
			dyds
		\biggr\|_{L^q (\mathbb{R}^3)}\\
		=& o \bigl( t^{-3 (1-\frac1q) - 2} \bigr)
	\end{split}
	\]
	as $t \to \infty$ for $1 \le q \le \infty$, where $M = \int_{\mathbb{R}^3} u_0 (y) dy$ and $m = \int_{\mathbb{R}^3} (-y) u_0 (y) dy$.
\end{theorem}
If we try to give the asymptotic expansion for the case $n = 2$ and $\theta = 1$ in the same way as above, then we may see that
$\int_0^\infty \int_{\mathbb{R}^2} |y| |u\nabla(-\Delta)^{-1}u - M^2 P\nabla(-\Delta)^{-1}P| dyds = + \infty$.
To study this case, we define 
\begin{equation}\label{J2}
\begin{split}
	J_2 (t)
	=&
	M\int_0^{t/2}
		\nabla P (t-s)*
		\bigl(
			P \nabla (-\Delta)^{-1} (m\cdot\nabla P)
			+
			(m\cdot\nabla P) \nabla (-\Delta)^{-1} P
		\bigr)
		(s)
	ds\\
	&+
	M \int_{t/2}^t
		P(t-s)*
		\nabla \cdot \bigl(
			P \nabla (-\Delta)^{-1} (m\cdot\nabla P)
			+
			(m\cdot\nabla P) \nabla (-\Delta)^{-1} P
		\bigr)
		(s)
	ds\\
	&+
	M^3 \int_0^{t/2} \int_{\mathbb{R}^2}
		\left( \nabla P (t-s,x-y) + (y\cdot\nabla) \nabla P(t,x) \right)\\
		&\hspace{15mm}\cdot
		\bigl(
			P \nabla (-\Delta)^{-1} J + J \nabla (-\Delta)^{-1} P
		\bigr)
		(s,y)
	dyds\\
	&+
	M^3\int_{t/2}^t
		P(t-s) * \nabla
		\bigl(
				P \nabla (-\Delta)^{-1} J + J \nabla (-\Delta)^{-1} P
		\bigr)
		(s)
	ds,\\
	K(t)
	=&
	\frac12 \Delta P(t) \log (1+\tfrac{t}2) \int_{\mathbb{R}^2}
		(-y) \cdot \left( P\nabla(-\Delta)^{-1}J + J \nabla (-\Delta)^{-1} P \right) (1,y)
	dy.
\end{split}
\end{equation}
Then $J_2$ satisfies
\begin{equation}\label{J2p}
	J_2 \in C \left( (0,\infty), L^1 (\mathbb{R}^2) \cap L^\infty (\mathbb{R}^2) \right),
	\quad
	\bigl\| J_2(t) \bigr\|_{L^q (\mathbb{R}^2)}
	=
	t^{-2 (1-\frac1q) - 2} \bigl\| J_2 (1) \bigr\|_{L^q (\mathbb{R}^2)}
\end{equation}
for $1 \le q \le \infty$ (see Proposition \ref{propJ2} in Section \ref{prf}).
\begin{theorem}\label{th4}
	Let $n = 2,~ \theta = 1,~ u_0 \in L^1 (\mathbb{R}^2,(1+|x|^2)dx) \cap L^\infty (\mathbb{R}^2)$, and $J,~ J_2$ and $K$ be given by \eqref{J} and \eqref{J2}.
	Assume that the solution $u$ satisfies \eqref{7} and \eqref{8}.
	Then
	\[
	\begin{split}
		\biggl\|
			&u(t) - M P (t) - m \cdot \nabla P (t) - M^2 J (t) - M^3 K(t) - J_2 (t)
			- \sum_{|\alpha| = 2} \frac{\nabla^\alpha P (t)}{\alpha!} 
			\int_{\mathbb{R}^2} (-y)^\alpha u_0 (y) dy\\
			&- \sum_{|\beta| = 1} \nabla^\beta \nabla P (t)
			\cdot \int_0^\infty \int_{\mathbb{R}^2}
				(-y)^\beta
				\bigl\{
					u\nabla(-\Delta)^{-1}u (s,y)
					- M^2 P\nabla(-\Delta)^{-1} P (s,y)\\
					&\hspace{15mm}- M\left( P \nabla (-\Delta)^{-1} (m \cdot \nabla P + M^2 J) + (m\cdot\nabla P + M^2 J) \nabla (-\Delta)^{-1} P \right) (1+s,y)
				\bigr\}
			dyds
		\biggr\|_{L^q (\mathbb{R}^2)}\\
		=& o \bigl( t^{-2 (1-\frac1q) - 2} \bigr)
	\end{split}
	\]
	as $t \to \infty$ for $1 \le q \le \infty$, where $M = \int_{\mathbb{R}^2} u_0 (y) dy$ and $m = \int_{\mathbb{R}^2} (-y) u_0 (y) dy$.
\end{theorem}
We confirm that
\begin{equation}\label{coef4}
\begin{split}
	&\int_0^\infty \int_{\mathbb{R}^2}
				(-y)^\beta
				\bigl\{
					u\nabla(-\Delta)^{-1}u (s,y)
					- M^2 P\nabla(-\Delta)^{-1} P (s,y)\\
					&\hspace{5mm}- M\left( P \nabla (-\Delta)^{-1} (m \cdot \nabla P + M^2 J) + (m\cdot\nabla P + M^2 J) \nabla (-\Delta)^{-1} P \right) (1+s,y)
				\bigr\}
			dyds
	\in \mathbb{R}^2
\end{split}
\end{equation}
in Section \ref{prf}.
Theorem \ref{th4} provides the asymptotic expansion with third-order.
Clearly, we see that the asymptotic expansion with second-order contains no logarithmic term.
Now we refer to the following generalized Burgers equation:
\begin{equation}\label{B}
	\left\{
	\begin{array}{lr}
		\partial_t \omega + (-\partial_x^2)^{1/2} \omega + \frac12 \partial_x (\omega^2) = 0,
		&
		t > 0,~ x \in \mathbb{R},\\
		\omega (0,x) = \omega_0 (x),
		&
		x \in \mathbb{R}.
	\end{array}
	\right.
\end{equation}
For \eqref{B}, well-posedness, global existence and decay of solutions for small initial data are proved.
Particularly, for $1 \le q \le \infty$, the decaying solution has the following asymptotic expansion as $t\to\infty$ (see \cite{I,Y-S}):
\begin{equation}\label{asympB}
	\begin{split}
	&\biggl\|
		\omega (t) - M_\omega P(t) + \frac1{4\pi} M_\omega^2 \partial_x P(t) \log (1+\tfrac{t}2) - M_\omega^2 J_\omega (t)\\
	&\hspace{2mm}-
		\biggl( m_\omega - \frac12 \int_0^\infty \int_{\mathbb{R}} \left( \omega(s,y)^2 - M_\omega^2 P(1+s,y)^2 \right) dyds \biggr) \partial_x P(t) \biggr\|_{L^q (\mathbb{R})}
	= o \bigl( t^{-(1-\frac1q)-1} \bigr),
	\end{split}
\end{equation}
where $M_\omega = \int_{\mathbb{R}} \omega_0 (y) dy,~ m_\omega = \int_{\mathbb{R}} (-y) \omega_0 (y) dy$ and
\[
\begin{split}
	J_\omega (t)
	=&
	-\frac12 \int_0^{t/2} \int_{\mathbb{R}}
		\left( \partial_x P(t-s,x-y) - \partial_x P(t,x) \right)
		P(s,y)^2
	dyds\\
	&-
	\int_{t/2}^t
		P (t-s) * (P\partial_x P)(s)
	ds.
\end{split}
\]
This correction term fulfills
\[
	\bigl\| J_\omega (t) \bigr\|_{L^q (\mathbb{R})}
	=
	t^{-(1-\frac1q)-1} \bigl\| J_\omega (1) \bigr\|_{L^q (\mathbb{R})}
\]
for $1 \le q \le \infty$.
The logarithmic term in \eqref{asympB} is derived from the following procedure:
The mild solution of \eqref{B} is given by
\[
	\omega (t) = P(t) * \omega_0 - \frac12 \int_0^{t/2} \partial_x P(t-s) * \omega (s)^2 ds - \int_{t/2}^t P(t-s) * (\omega\partial_x \omega) (s) ds.
\]
In the second term, we renormalize $\omega$ by $M_\omega P$, then we obtain the term
$\frac12 M_\omega^2 \int_0^{t/2}
	\partial_x P(t-s) * P (1+s)^2
ds$.
Taylor's theorem says that the decay rate of this term is given by
\[
\begin{split}
	\frac12 M_\omega^2 \partial_x P(t) \int_0^{t/2} \int_{\mathbb{R}} P (1+s,y)^2 dyds
	=&
	\frac12 M_\omega^2 \partial_x P(t) \int_0^{t/2} (1+s)^{-1} ds \int_{\mathbb{R}} P (1,y)^2 dy\\
	=&
	\frac1{4\pi} M_\omega^2 \partial_x P(t) \log ( 1+\tfrac{t}2 ).
\end{split}
\]
Here we used the relation $P(1+s,y) = (1+s)^{-2} P(1,(1+s)^{-1} y)$.
Similarly, the second-order asymptotic expansion for \eqref{DD} with $n = 2$ and $\theta =1$ contains
\[
\begin{split}
	&M^2 \nabla P(t) \cdot \int_0^{t/2} \int_{\mathbb{R}^2}
		(P\nabla(-\Delta)^{-1}P) (1+s,y)
	dyds\\
	=&
	M^2 \nabla P(t) \log (1+\tfrac{t}2) \cdot \int_{\mathbb{R}^2} (P\nabla(-\Delta)^{-1}P) (1,y) dy,
\end{split}
\]
since $P(1+s,y) = (1+s)^{-3} P(1,(1+s)^{-1}y)$ when $n = 2$.
This fact does not contradict the assertion of Theorem \ref{th4}.
Indeed
\[
	\int_{\mathbb{R}^2} (P\nabla(-\Delta)^{-1}P) (1,y) dy = 0.
\]
Such a vanishing logarithmic term is developed in the studies for some other phenomena (we refer to \cite{H-K-N,H-N,K-RP,KtM,N-Y,N-S_Mono,N-S_DE,Yd}).\\

\noindent
{\bf Notation.}
In this paper, we use the following notation.
For $a = (a_1,\ldots,a_n)$ and $b = (b_1,\ldots,b_n) \in \mathbb{R}^n$, we denote that $a \cdot b = \sum_{j=1}^n a_j b_j$ and $|a| = \sqrt{a\cdot a}$.
We define the Fourier transform and the Fourier inverse transform by $\mathcal{F}[\varphi](\xi) = (2\pi)^{-n/2} \int_{\mathbb{R}^n} e^{-i x\cdot \xi} \varphi (x) dx$ and $\mathcal{F}^{-1} [\varphi] (x) = (2\pi)^{-n/2} \int_{\mathbb{R}^n} e^{ix\cdot \xi} \varphi (\xi) d\xi$, where $i = \sqrt{-1}$.
We denote that $\partial_t = \partial / \partial t,~ \partial_j = \partial / \partial x_j~ (j = 1,\ldots,n),~ \nabla = (\partial_1,\ldots,\partial_n)$ and $\Delta = \sum_{j=1}^n \partial_j^2$.
Particularly $\partial_x = \partial/\partial x$ for $n = 1$, and $\nabla^\bot = (-\partial_2, \partial_1)$ for $n = 2$.
For $\theta > 0,~ (-\Delta)^{\theta/2} \varphi = \mathcal{F}^{-1} [|\xi| \mathcal{F} [\varphi]]$.
For $\alpha = (\alpha_1,\ldots,\alpha_n) \in \mathbb{Z}_+^n = (\mathbb{N} \cup \{ 0 \})^n$, we use $\alpha! = \prod_{j=1}^n \alpha_j !,~ \nabla^\alpha = \prod_{j=1}^n \partial_j^{\alpha_j}$ and $|\alpha| = \sum_{j=1}^n \alpha_j$.
For $1 \le p \le \infty$ and $s \in \mathbb{R},~ L^p (\mathbb{R}^n)$ and $W^{s,p} (\mathbb{R}^n)$ denote the Lebesgue space and the Sobolev space on $\mathbb{R}^n$, respectively.
We abbreviate the norm of $L^p (\mathbb{R}^n)$ by $\| \cdot \|_{L^p (\mathbb{R}^n)}$.
For a nonnegative function $g$, let $L^1 (\mathbb{R}^n, g dx) = \{ \varphi \in L_\mathit{loc}^1 (\mathbb{R}^n)~ |~ \int_{\mathbb{R}^n} |\varphi (x)| g(x) dx < + \infty \}$.
We write the convolution of $f = f(x)$ and $g = g(x)$ by $f*g (x) = \int_{\mathbb{R}^n} f(x-y) g(y) dy$.
The gamma function is provided by $\Gamma (p) = \int_0^\infty e^{-t} t^{p-1} dt$ for $p > 0$.
Various constants are simply denoted by $C$.

\section{Preliminaries}\label{2}
In this section, we prepare several lemmas to use in the proof of our results.
\begin{lemma}[positivity lemma]\label{lemC-C}
	Let $0 \le s \le 2,~ p \ge 1$ and $f \in W^{s,p} (\mathbb{R}^n)$.
	Then
	\[
		\int_{\mathbb{R}^n}
			|f|^{p-2} f (-\Delta)^{s/2} f
		dx
		\ge 0.
	\]
	Particularly, when $p \ge 2$,
	\[
		\int_{\mathbb{R}^n}
			|f|^{p-2} f (-\Delta)^{s/2} f
		dx
		\ge
		\frac2p \int_{\mathbb{R}^n}
			\left| (-\Delta)^{s/4} (|f|^{p/2}) \right|^2
		dx
	\]
	holds.
\end{lemma}
For the proof of this lemma, see \cite{Crdb-Crdb,J05}.
We also need some inequalities of Sobolev type.
\begin{lemma}[Hardy-Littlewood-Sobolev's inequality \cite{S,Z}]\label{HLS}
	Let $n \ge 2,~ 1 < \sigma < n,~ 1 < p < \frac{n}{\sigma}$ and $\frac1{p_*} = \frac1p - \frac{\sigma}n$.
	Then there exists a positive constant $C$ such that
	\[
		\bigl\| (-\Delta)^{-\sigma/2} \varphi \bigr\|_{L^{p_*} (\mathbb{R}^n)}
		\le
		C \bigl\| \varphi \bigr\|_{L^p (\mathbb{R}^n)}
	\]
	for any $\varphi \in L^p (\mathbb{R}^n)$.
\end{lemma}
\begin{lemma}[Gagliardo-Nirenberg inequality \cite{H-Y-Z,K-PP,M-N-S-S}]\label{GN}
	Let $n \ge 1,~ 0 < \sigma < s < n,~ 1 < p_1, p_2 < \infty$ and $\frac1p = (1-\frac\sigma{s}) \frac1{p_1} + \frac\sigma{s} \frac1{p_2}$.
	Then
	\[
		\bigl\| (-\Delta)^{\sigma/2} \varphi \bigr\|_{L^p (\mathbb{R}^n)}
		\le
		C \bigl\| \varphi \bigr\|_{L^{p_1} (\mathbb{R}^n)}^{1-\frac\sigma{s}}
		\bigl\| (-\Delta)^{s/2} \varphi \bigr\|_{L^{p_2} (\mathbb{R}^n)}^{\frac\sigma{s}}
	\]
	holds.
\end{lemma}
The following estimate is due to \cite{K-P}.
\begin{lemma}[Kato-Ponce's commutator estimates \cite{J04,K-P}]\label{cmm}
	Let $s > 0$ and $1 < p < \infty$.
	Then
	\[
		\bigl\| [(-\Delta)^{s/2}, g] f \bigr\|_{L^p (\mathbb{R}^n)}
		\le
		C \bigl( \| \nabla g \|_{L^{p_1} (\mathbb{R}^n)} \| (-\Delta)^{(s-1)/2} f \|_{L^{p_2} (\mathbb{R}^n)}
		+ \| (-\Delta)^{s/2} g \|_{L^{p_3} (\mathbb{R}^n)} \| f \|_{L^{p_4} (\mathbb{R}^n)} \bigr)
	\]
	and
	\[
		\bigl\| (-\Delta)^{s/2} (fg) \bigr\|_{L^p (\mathbb{R}^n)}
		\le
		C \bigl( \| f \|_{L^{p_1} (\mathbb{R}^n)} \| (-\Delta)^{s/2} g \|_{L^{p_2} (\mathbb{R}^n)}
		+
		\| (-\Delta)^{s/2} f \|_{L^{p_3} (\mathbb{R}^n)} \| g \|_{L^{p_4} (\mathbb{R}^n)} \bigr)
	\]
	with $1 < p_j \le \infty~ (j = 1,4)$ and $1 < p_j < \infty~ (j = 2,3)$ such that $\frac1p = \frac1{p_1} + \frac1{p_2} = \frac1{p_3} + \frac1{p_4}$.
\end{lemma}
The H\"ormander-Mikhlin type inequality (cf. \cite[Theorem 3.1]{S-S}) yields that
\[
	\bigl| \partial_t^m \nabla^\alpha G_\theta (1,x) \bigr|
	\le
	C (1+|x|)^{-n-\theta-\theta m - |\alpha|}
\]
for $m \in \mathbb{Z}_+$ and $\alpha \in \mathbb{Z}_+^n$.
A coupling of this and the scaling property
\[
	\partial_t^m \nabla^\alpha G_\theta (t,x)
	=
	t^{-\frac{n}\theta - m - \frac{|\alpha|}\theta}
	\partial_t^m \nabla^\alpha G_\theta (1,x)
\]
provides the following lemmas.
\begin{lemma}\label{decay-lin}
	Let $n \ge 1,~ \theta > 0,~ m \in \mathbb{Z}_+,~ \alpha \in \mathbb{Z}_+^n$ and $1 \le p \le q \le \infty$.
	Then there exists a positive constant $C$ such that
	\[
		\bigl\| \partial_t^m \nabla^\alpha G_\theta (t) * \varphi \bigr\|_{L^q (\mathbb{R}^n)}
		\le
		C t^{-\frac{n}\theta (\frac1p - \frac1q) - m - \frac{|\alpha|}\theta} \bigl\| \varphi \bigr\|_{L^p (\mathbb{R}^n)}
	\]
	for any $\varphi \in L^p (\mathbb{R}^n)$.
\end{lemma}
\begin{lemma}\label{ap-lin}
	Let $N \in \mathbb{Z}_+,~ \varphi \in L^1 (\mathbb{R}^n,(1+|x|)^{N} dx)$ and $1 \le q \le \infty$.
	Then
	\[
		\biggl\| G_\theta (t) * \varphi - \sum_{|\alpha| \le N} \frac{\nabla^\alpha G_\theta (t)}{\alpha!} \int_{\mathbb{R}^n} (-y)^\alpha \varphi (y) dy
		\biggr\|_{L^q (\mathbb{R}^n)}
		=
		o \bigl( t^{-\frac{n}\theta (1-\frac1q) - \frac{N}\theta} \bigr)
	\]
	as $t\to\infty$.
	In addition, if $\varphi \in L^1 (\mathbb{R}^n, (1+|x|^2)^{(N+1)/2} dx)$, then
	\[
		\biggl\| |x|^\mu \biggl( G_\theta (t) * \varphi - \sum_{|\alpha| \le N} \frac{\nabla^\alpha G_\theta (t)}{\alpha!} \int_{\mathbb{R}^n} (-y)^\alpha \varphi (y) dy \biggr)
		\biggr\|_{L^q (\mathbb{R}^n)}
		\le
		C t^{-\frac{n}\theta (1-\frac1q) - \frac{N+1}\theta + \frac\mu\theta}
	\]
	for $0 \le \mu \le N$ and $t>0$.
\end{lemma}
The solution of \eqref{DD} satisfies the following estimate.
\begin{proposition}\label{lem-dr}
	Let $n \ge 2,~ 0 < \theta \le 1$ and $\sigma \ge 0$.
	Assume that the solution $u$ satisfies \eqref{7} and \eqref{8}.
	Then there exist positive constants $C$ and $T$ such that
\[
		\bigl\| (-\Delta)^{\sigma/2} u(t) \bigr\|_{L^2 (\mathbb{R}^n)}
		\le
		C t^{-\frac{n}{2\theta} -\frac\sigma\theta}
\]
	for any $t \ge T$.
\end{proposition}
\begin{proof}
Let $q > \frac{n}\theta + \frac{2\sigma}\theta$.
Using \eqref{DD}, we see that
\[
\begin{split}
	&\frac12 \frac{d}{dt} \left( t^q \| (-\Delta)^{\sigma/2} u(t) \|_{L^2 (\mathbb{R}^n)}^2 \right)
	+
	t^q \| (-\Delta)^{\frac\sigma2 + \frac\theta4} u (t) \|_{L^2 (\mathbb{R}^n)}^2\\
	=&
	t^q \int_{\mathbb{R}^n}
		(-\Delta)^{\sigma/2} u \nabla (-\Delta)^{\sigma/2} \cdot (u\nabla\psi)
	dx
	+\frac{q}2 t^{q-1} \| (-\Delta)^{\sigma/2} u(t) \|_{L^2 (\mathbb{R}^n)}^2.
\end{split}
\]
Since
\[
\begin{split}
	&\int_{\mathbb{R}^n}
		(-\Delta)^{\sigma/2} u \nabla (-\Delta)^{\sigma/2} \cdot (u\nabla\psi)
	dx\\
	=&
	\int_{\mathbb{R}^n}
		(-\Delta)^{\sigma/2} u \nabla (-\Delta)^{\sigma/2} u \cdot \nabla \psi
	dx
	+
	\int_{\mathbb{R}^n}
		(-\Delta)^{\sigma/2} u
		\bigl[ \nabla(-\Delta)^{\sigma/2}, \nabla\psi \bigr] u
	dx\\
	=&
	\frac12 \int_{\mathbb{R}^n}
		u \bigl| (-\Delta)^{\sigma/2} u \bigr|^2
	dx
	+
	\int_{\mathbb{R}^n}
		(-\Delta)^{\sigma/2} u
		\bigl[ \nabla(-\Delta)^{\sigma/2}, \nabla\psi \bigr] u
	dx,
\end{split}
\]
we have that
\begin{equation}\label{bs}
\begin{split}
	&\frac12 \frac{d}{dt} \left( t^q \| (-\Delta)^{\sigma/2} u(t) \|_{L^2 (\mathbb{R}^n)}^2 \right)
	+
	t^q \| (-\Delta)^{\frac\sigma2 + \frac\theta4} u (t) \|_{L^2 (\mathbb{R}^n)}^2\\
	=&
	\frac12 t^q \int_{\mathbb{R}^n}
		u \left| (-\Delta)^{\sigma/2} u \right|^2
	dx
	+
	t^q \int_{\mathbb{R}^n}
		(-\Delta)^{\sigma/2} u
		\bigl[ \nabla(-\Delta)^{\sigma/2}, \nabla\psi \bigr] u
	dx\\
	&+
	\frac{q}2 t^{q-1} \| (-\Delta)^{\sigma/2} u(t) \|_{L^2 (\mathbb{R}^n)}^2.
\end{split}
\end{equation}
Let $\frac1\rho = \frac12 - \frac\theta{2n}$, then, from \eqref{7}, we see that
\[
\begin{split}
	\int_{\mathbb{R}^n}
		u \left| (-\Delta)^{\sigma/2} u \right|^2
	dx
	\le&
	\bigl\| u \bigr\|_{L^{\rho/(\rho-2)} (\mathbb{R}^n)} \bigl\| (-\Delta)^{\sigma/2} u \bigr\|_{L^\rho (\mathbb{R}^n)}^2
	\le
	C t^{-\frac{n}\theta (1-\frac\theta{n})} \bigl\| (-\Delta)^{\frac\sigma{2} + \frac\theta{4}} u \bigr\|_{L^2 (\mathbb{R}^n)}^2\\
	\le&
	\frac14\bigl\| (-\Delta)^{\frac\sigma{2} + \frac\theta{4}} u \bigr\|_{L^2 (\mathbb{R}^n)}^2
\end{split}
\]
for sufficiently large $t$.
The H\"older inequality yields that
\[
	\int_{\mathbb{R}^n}
		(-\Delta)^{\sigma/2} u
		\bigl[ \nabla(-\Delta)^{\sigma/2}, \nabla\psi \bigr] u
	dx
	\le
	\bigl\| (-\Delta)^{\sigma/2} u \bigr\|_{L^\rho (\mathbb{R}^n)}
	\bigl\| \bigl[ \nabla(-\Delta)^{\sigma/2}, \nabla\psi \bigr] u \bigr\|_{L^{\rho'} (\mathbb{R}^n)},
\]
where $\frac1\rho = \frac12 - \frac\theta{2n}$ and $\frac1{\rho'} = \frac12 + \frac\theta{2n}$.
Using Lemma \ref{cmm} and \eqref{7}, we see that
\[
\begin{split}
	&\bigl\|
		\bigl[ \nabla(-\Delta)^{\sigma/2}, \nabla\psi \bigr] u
	\bigr\|_{L^{\rho'} (\mathbb{R}^n)}\\
	\le&
	C \left( \| (-\Delta)^{\frac{\sigma+1}2} \nabla \psi \|_{L^\rho (\mathbb{R}^n)} \| u \|_{L^{n/\theta} (\mathbb{R}^n)}
	+
	\| \nabla^2\psi \|_{L^{n/\theta} (\mathbb{R}^n)} \| (-\Delta)^{\sigma/2} u \|_{L^\rho (\mathbb{R}^n)} \right)\\
	\le&
	C \| u \|_{L^{n/\theta} (\mathbb{R}^n)} \| (-\Delta)^{\sigma/2} u \|_{L^\rho (\mathbb{R}^n)}
	\le
	C (1+t)^{-\frac{n}\theta+1} \| (-\Delta)^{\sigma/2} u \|_{L^\rho (\mathbb{R}^n)}.
\end{split}
\]
The Sobolev inequality says that
\[
	\| (-\Delta)^{\sigma/2} u \|_{L^\rho (\mathbb{R}^n)}
	\le
	C \| (-\Delta)^{\frac\sigma{2} + \frac\theta{4}} u \|_{L^2 (\mathbb{R}^n)}.
\]
Thus we have that
\[
\begin{split}
	\int_{\mathbb{R}^n}
		(-\Delta)^{\sigma/2} u
		\bigl[ \nabla(-\Delta)^{\sigma/2}, \nabla\psi \bigr] u
	dx
	\le&
	C (1+t)^{-\frac{n}\theta+1} \| (-\Delta)^{\frac\sigma{2}+\frac\theta{4}} u \|_{L^2 (\mathbb{R}^n)}^2\\
	\le&
	\frac18 \| (-\Delta)^{\frac\sigma{2}+\frac\theta{4}} u \|_{L^2 (\mathbb{R}^n)}^2
\end{split}
\]
for sufficiently large $t$.
The third term on the right-hand side of \eqref{bs} is treated by Lemma \ref{GN}.
Namely, for $\lambda = \frac{2\sigma}{2\sigma + \theta}$, we see that
\[
\begin{split}
	t^{q-1} \bigl\| (-\Delta)^{\sigma/2} u (t) \bigr\|_{L^2 (\mathbb{R}^n)}^2
	\le&
	t^{q-1} \bigl\| u(t) \bigr\|_{L^2 (\mathbb{R}^n)}^{2(1-\lambda)}
	\bigl\| (-\Delta)^{\frac\sigma{2} + \frac\theta{4}} u (t) \bigr\|_{L^2 (\mathbb{R}^n)}^{2\lambda}\\
	\le&
	C t^{q-1-\frac{n}\theta-\frac{2\sigma}\theta} \| u(t) \|_{L^2 (\mathbb{R}^n)}^2
	+
	\frac18 t^q \bigl\| (-\Delta)^{\frac\sigma{2} + \frac\theta{4}} u (t) \bigr\|_{L^2 (\mathbb{R}^n)}^2.
\end{split}
\]
Therefore we obtain that
\[
	\frac{d}{dt} \left(
		t^q \bigl\| (-\Delta)^{\sigma/2} u(t) \bigr\|_{L^2 (\mathbb{R}^n)}^2
	\right)
	+
	t^q \bigl\| (-\Delta)^{\frac\sigma{2} + \frac\theta{4}} u(t) \bigr\|_{L^2 (\mathbb{R}^n)}^2
	\le
	C t^{q-1-\frac{n}\theta-\frac{2\sigma}\theta}
\]
for large $t$.
If we choose sufficiently large $T$, then, we conclude that
\[
\begin{split}
	&t^q \bigl\| (-\Delta)^{\sigma/2} u(t) \bigr\|_{L^2 (\mathbb{R}^n)}^2
	+
	\int_T^t
		s^q \bigl\| (-\Delta)^{\frac\sigma{2} + \frac\theta{4}} u(s)\bigr\|_{L^2 (\mathbb{R}^n)}^2
	ds\\
	\le&
	T^q \bigl\| (-\Delta)^{\sigma/2} u(T) \bigr\|_{L^2 (\mathbb{R}^n)}^2
	+
	C \int_T^t s^{q-1-\frac{n}\theta-\frac{2\sigma}\theta} ds
\end{split}
\]
for $t \ge T$, and complete the proof.
\end{proof}
The decay of the conservation force field $\nabla \psi$ is given in the following.
\begin{proposition}\label{HLSinfty}
	Upon \eqref{7}, $\nabla \psi = \nabla (-\Delta)^{-1} u$ on \eqref{DD} fulfills \eqref{decay-pot}
	for $\frac{n}{n-1} < p \le \infty$.
\end{proposition}
\begin{proof}
Lemma \ref{HLS} and \eqref{7} give the assertion for $\frac{n}{n-1} < p < \infty$.
Since
\[
	\nabla (-\Delta)^{-1} \varphi(x)
	=
	\frac{\Gamma (\tfrac{n}2)}{2 \pi^{\tfrac{n}2}} \int_{\mathbb{R}^n} \frac{x-y}{|x-y|^{n}} \varphi(y) dy,
\]
we see
\[
\begin{split}
	\left| \nabla (-\Delta)^{-1} u(t) \right|
	\le&
	C \biggl( \int_{|x-y| \le (1+t)^{1/\theta}} + \int_{|x-y| \ge (1+t)^{1/\theta}} \biggr)
		\frac{|u(t,y)|}{|x-y|^{n-1}}
	dy\\
	\le&
	C \left( (1+t)^{\frac1\theta} \| u(t) \|_{L^\infty (\mathbb{R}^n)}
	+
	(1+t)^{-\frac{n}\theta+\frac1\theta} \| u(t) \|_{L^1 (\mathbb{R}^n)} \right).
\end{split}
\]
This inequality together with \eqref{7} leads the assertion for $p = \infty$.
\end{proof}
The moment of the solution fulfills the following estimate.
\begin{proposition}\label{lem-wt}
	Let $n \ge 2,~ 0 < \theta \le 1$ and the solution $u$ of \eqref{DD} satisfy \eqref{7}.
	Assume that $x u_0 \in L^{n/(n-1)} (\mathbb{R}^n)$.
	Then
	\[
		\bigl\| x_j u (t) \bigr\|_{L^{n/(n-1)} (\mathbb{R}^n)}
		\le
		C \log (e+t)
	\]
	for $j = 1,\ldots,n$.
\end{proposition}
\begin{proof}
Let $p = \frac{n}{n-1}$.
Multiplying the first equation in \eqref{DD} by $x_j |x_j u|^{p-2} x_j u$ and integrate over $\mathbb{R}^n$, we have that
\begin{equation}\label{lem-wt-bs}
\begin{split}
	&\frac1p \frac{d}{dt} \bigl\| x_j u \bigr\|_{L^p (\mathbb{R}^n)}^p
	+
	\int_{\mathbb{R}^n}
		\bigl| x_j u \bigr|^{p-2} x_j u (-\Delta)^{\theta/2} (x_j u) dx\\
	=&
	-\int_{\mathbb{R}^n}
		|x_j u|^{p-2} x_j u
		\bigl[ x_j, (-\Delta)^{\theta/2} \bigr] u
	dx
	+
	\int_{\mathbb{R}^n}
		x_j |x_j u|^{p-2} x_j u \nabla u \cdot \nabla (-\Delta)^{-1} u
	dx
	-
	\int_{\mathbb{R}^n}
		|x_j u|^p u
	dx.
\end{split}
\end{equation}
Lemma \ref{lemC-C} implies the positivity of the second term in the left hand side of the above equality.
The relation $[x_j, (-\Delta)^{\theta/2}] = \theta (-\Delta)^{\frac{\theta-2}2} \partial_j$, the H\"older inequality, and Hardy-Littlewood-Sobolev's inequality together with \eqref{7} provide that
\[
\begin{split}
	\biggl| \int_{\mathbb{R}^n}
		|x_j u|^{p-2} x_j u
		\bigl[ x_j, (-\Delta)^{\theta/2} \bigr] u
	dx \biggr|
	\le&
	C \bigl\| (-\Delta)^{\frac{\theta-2}2} \partial_j u \bigr\|_{L^p (\mathbb{R}^n)}
	\bigl\| x_j u \bigr\|_{L^p (\mathbb{R}^n)}^{p-1}\\
	\le&
	C (1+t)^{-1} \bigl\| x_j u \bigr\|_{L^p (\mathbb{R}^n)}^{p-1}.
\end{split}
\]
%
Similarly we obtain that
\[
\begin{split}
	&\int_{\mathbb{R}^n}
		x_j |x_j u|^{p-2} x_j u \nabla u \cdot \nabla (-\Delta)^{-1} u
	dx
	-
	\int_{\mathbb{R}^n}
		|x_j u|^p u
	dx\\
	=&
	\int_{\mathbb{R}^n}
		|x_j u|^{p-2} x_j u
		\nabla (x_j u)
		\cdot
		\nabla (-\Delta)^{-1} u
	dx
	-
	\int_{\mathbb{R}^n}
		|x_j u|^p u
	dx
	+
	\int_{\mathbb{R}^n}
		|x_j u|^{p-2} x_j u
		\bigl[ x_j, \nabla \bigr] u
		\cdot
		\nabla (-\Delta)^{-1} u
	dx\\
	=&
	\left( \tfrac1p - 1 \right) \int_{\mathbb{R}^n}
		|x_j u|^p u
	dx
	+
	\int_{\mathbb{R}^n}
		u \partial_j (-\Delta)^{-1} u |x_j u|^{p-2} x_j u
	dx.
\end{split}
\]
The H\"older inequality, the Sobolev inequality and \eqref{7} yield that
\[
\begin{split}
	&\biggl|
		\int_{\mathbb{R}^n}
		x_j |x_j u|^{p-2} x_j u \nabla u \cdot \nabla (-\Delta)^{-1} u
	dx
	-
	\int_{\mathbb{R}^n}
		|x_j u|^p u
	dx
	\biggr|\\
	\le&
	C \bigl( \bigl\| u \bigr\|_{L^\infty (\mathbb{R}^n)} \bigl\| x_j u \bigr\|_{L^p (\mathbb{R}^n)}^p
	+
	\bigl\| u \partial_j (-\Delta)^{-1} u \bigr\|_{L^p (\mathbb{R}^n)}
	\bigl\| x_j u \bigr\|_{L^p (\mathbb{R}^n)}^{p-1} \bigr)\\
	\le&
	C (1+t)^{-n/\theta} \bigl( 1 + \bigl\| x_j u \bigr\|_{L^p (\mathbb{R}^n)} \bigr)
	\bigl\| x_j u \bigr\|_{L^p (\mathbb{R}^n)}^{p-1}\\
	\le&
	C \left( (1+t)^{-1} \bigl\| x_j u \bigr\|_{L^p (\mathbb{R}^n)}^{p-1} + (1+t)^{-n/\theta} \bigl\| x_j u \bigr\|_{L^p (\mathbb{R}^n)}^p \right).
\end{split}
\]
Therefore, from \eqref{lem-wt-bs}, we obtain the relation
\[
	f' (t) \le C_0 \left( (1+t)^{-1} f(t)^{1/n} + (1+t)^{-n/\theta} f(t) \right)
\]
for $f(t) = \| x_j u (t) \|_{L^p (\mathbb{R}^n)}^p$.
Let $g(t) = \exp (-C_0 \int_0^t (1+s)^{-n/\theta} ds)$.
Then there exists a positive constant $\varepsilon > 0$ such that $\varepsilon \le g(t) \le \varepsilon^{-1}$ for any $t$, and we see that
\[
	\left( f(t) g(t) \right)'
	\le
	C_0 (1+t)^{-1} f(t)^{1/n} g(t)
	\le
	C (1+t)^{-1} \left( f(t) g(t) \right)^{1/n}.
\]
Solving this inequality, we complete the proof.
\end{proof}
Since $L^1 (\mathbb{R}^n, (1+|x|)^2 dx) \cap L^\infty (\mathbb{R}^n) \subset L^{\frac{n}{n-1}} (\mathbb{R}^n, (1+|x|) dx)$, the assertion of Proposition \ref{lem-wt} is satisfied upon the assumption of our main theorems.
Before closing this section, we show the asymptotic profile of the solution.
\begin{proposition}\label{ap}
	Let $n \ge 2,~ 0 < \theta \le 1,~ 1 \le q \le \infty,~ u_0 \in L^1 (\mathbb{R}^n) \cap L^\infty (\mathbb{R}^n)$, and the solution $u$ of \eqref{DD} fulfill \eqref{7} and \eqref{8}.
	Then
	\[
		\left\| u(t) - M G_\theta (t) \right\|_{L^q (\mathbb{R}^n)}
		=
		o \bigl( t^{-\frac{n}\theta (1-\frac1p)} \bigr)
	\]
	as $t\to\infty$, where $M = \int_{\mathbb{R}^n} u_0 (y) dy$.
	In addition, if $xu_0 \in L^1 (\mathbb{R}^n)$, then
	\[
		\left\| u(t) - M G_\theta (t) \right\|_{L^q (\mathbb{R}^n)}
		\le
		\left\{
	 	\begin{array}{lr}
		C t^{-\frac{n}\theta (1-\frac1q)} (1+t)^{-\frac1\theta}
		&
		(n \ge 3~ \text{or}~ \theta < 1)\\
		C t^{-2(1-\frac1q)} (1+t)^{-1} \log (e+t)
		&
		(n = 2~ \text{and}~ \theta = 1)
		\end{array}
		\right.
	\]
	for $t > 0$.
\end{proposition}
\begin{proof}
By \eqref{MS}, we see that
\begin{equation}\label{ap-bs}
\begin{split}
	&u(t) - MG_\theta (t)
	=
	G_\theta (t) * u_0 - MG_\theta (t)
	+
	\int_0^t \nabla G_\theta (t-s) * (u\nabla(-\Delta)^{-1}u) (s) ds.
\end{split}
\end{equation}
Since the estimate for the linear part is well-known, we consider the nonlinear term.
By Lemmas \ref{decay-lin} and \ref{HLS}, and \eqref{7}, we have that
\[
\begin{split}
	&\biggl\|
	\int_0^{t/2}
		\nabla G_\theta (t-s) * (u\nabla(-\Delta)^{-1}u) (s)
	ds
	\biggr\|_{L^q (\mathbb{R}^n)}\\
	\le&
	C \int_0^{t/2}
		(t-s)^{-\frac{n}\theta (1-\frac1q) - \frac1\theta}
		\| u \nabla (-\Delta)^{-1} u \|_{L^1 (\mathbb{R}^n)}
	ds\\
	\le&
	C \int_0^{t/2}
		(t-s)^{-\frac{n}\theta (1-\frac1q) - \frac1\theta}
		(1+s)^{-\frac{n-1}\theta}
	ds.
\end{split}
\]
Thus
\begin{equation}\label{ap-bs1}
\begin{split}
	&\biggl\|
	\int_0^{t/2}
		\nabla G_\theta (t-s) * (u\nabla(-\Delta)^{-1}u) (s)
	ds
	\biggr\|_{L^q (\mathbb{R}^n)}\\
	\le&
		\left\{
	 	\begin{array}{lr}
		C t^{-\frac{n}\theta (1-\frac1q)} (1+t)^{-\frac1\theta}
		&
		(n \ge 3~ \text{or}~ \theta < 1)\,\\
		C t^{-2(1-\frac1q)} (1+t)^{-1} \log (e+t)
		&
		(n = 2~ \text{and}~ \theta = 1).
		\end{array}
		\right.
\end{split}
\end{equation}
When $n \ge 3$, we obtain by Lemmas \ref{decay-lin} and \ref{HLS}, Proposition \ref{lem-dr}, and \eqref{7} that
\[
\begin{split}
	&\biggl\|
	\int_{t/2}^t
		G_\theta (t-s) * \nabla \cdot (u\nabla(-\Delta)^{-1}u) (s)
	ds
	\biggr\|_{L^1 (\mathbb{R}^n)}\\
	\le&
	C \int_{t/2}^t
	\left( \| \nabla u \|_{L^2 (\mathbb{R}^n)} \| \nabla (-\Delta)^{-1} u \|_{L^2 (\mathbb{R}^n)}
	+ \| u \|_{L^2 (\mathbb{R}^n)}^2 \right)
	ds\\
	\le&
	C \int_{t/2}^t
		(1+s)^{-\frac{n}\theta}
	ds.
\end{split}
\]
When $n = 2$, Lemma \ref{decay-lin} gives that
\[
\begin{split}
	&\biggl\|
	\int_{t/2}^t
		G_\theta (t-s) * \nabla \cdot (u\nabla(-\Delta)^{-1}u) (s)
	ds
	\biggr\|_{L^1 (\mathbb{R}^2)}\\
	\le&
	C \int_{t/2}^t
		\left( \bigl\| \nabla u \bigr\|_{L^{9/5} (\mathbb{R}^2)} \bigl\| \nabla (-\Delta)^{-1} u \bigr\|_{L^{9/4} (\mathbb{R}^2)}
		+
		\bigl\| u \bigr\|_{L^2 (\mathbb{R}^2)}^2 \right)
	ds.
\end{split}
\]
Here the Gagliardo-Nirenberg inequality yields that
\[
	\bigl\| \nabla u \bigr\|_{L^{9/5} (\mathbb{R}^2)}
	\le
	C \bigl\| u \bigr\|_{L^{3/2} (\mathbb{R}^2)}^{1/3} \bigl\| (-\Delta)^{3/4} u \bigr\|_{L^2(\mathbb{R}^2)}^{2/3}.
\]
Therefore Lemmas \ref{HLS} and \ref{lem-dr}, and \eqref{7} conclude that
\[
	\biggl\|
	\int_{t/2}^t
		G_\theta (t-s) * \nabla \cdot (u\nabla(-\Delta)^{-1}u) (s)
	ds
	\biggr\|_{L^1 (\mathbb{R}^2)}
	\le
	C \int_{t/2}^t (1+s)^{-\frac2\theta} ds.
\]
If we put $\sigma = \frac{(1-\varepsilon)n}{2}$ and $\frac1r = \frac\varepsilon{2}$ for some small $\varepsilon > 0$, then, by Lemma \ref{decay-lin} and the Sobolev inequality, we obtain that
\[
\begin{split}
	&\biggl\| 
	\int_{t/2}^t
		G_\theta (t-s) * \nabla \cdot (u\nabla(-\Delta)^{-1}u) (s)
	ds
	\biggr\|_{L^\infty (\mathbb{R}^n)}\\
	\le&
	C \int_{t/2}^t
		(t-s)^{-\frac{n}{\theta r}}
		\bigl(
		\bigl\| \nabla u \bigr\|_{L^r (\mathbb{R}^n)} \bigl\| \nabla (-\Delta)^{-1} u \bigr\|_{L^\infty (\mathbb{R}^n)}
		+
		\bigl\| u \bigr\|_{L^{2r} (\mathbb{R}^n)}^2
		\bigr)
	ds\\
	\le&
	\int_{t/2}^t
		(t-s)^{-\frac{\varepsilon n}{2\theta}}
		\bigl( \| \nabla (-\Delta)^{\sigma/2} u \|_{L^2 (\mathbb{R}^n)} \| \nabla (-\Delta)^{-1} u \|_{L^\infty (\mathbb{R}^n)}
		+
		\bigl\| u \bigr\|_{L^{2r} (\mathbb{R}^n)}^2
		\bigr)
	ds.
\end{split}
\]
Hence, by Propositions \ref{lem-dr} and \ref{HLSinfty}, and \eqref{7}, we obtain that
\[
\begin{split}
	&\biggl\| 
	\int_{t/2}^t
		G_\theta (t-s) * \nabla \cdot (u\nabla(-\Delta)^{-1}u) (s)
	ds
	\biggr\|_{L^\infty (\mathbb{R}^n)}
	\le
	C \int_{t/2}^t
		(t-s)^{-\frac{\varepsilon n}{2\theta}}
		(1+s)^{-\frac{2n}\theta + \frac{\varepsilon n}{2\theta}}
	ds.
\end{split}
\]
Thus, by the H\"older inequality, we conclude that
\begin{equation}\label{ap-bs4}
	\biggl\| 
	\int_{t/2}^t
		G_\theta (t-s) * \nabla \cdot (u\nabla(-\Delta)^{-1}u) (s)
	ds
	\biggr\|_{L^q (\mathbb{R}^n)}
	\le
	C (1+t)^{-\frac{n}\theta (1-\frac1q) - \frac{n}\theta + 1}
\end{equation}
for $1 \le q \le \infty$.
Applying \eqref{ap-bs1} and \eqref{ap-bs4} to \eqref{ap-bs}, we complete the proof.
\end{proof}
\section{Proof of main results}\label{prf}
\subsection{Proof of Theorem \ref{th1}}
In \eqref{MS}, large-time behavior of $G_\theta (t) * u_0$ is well-known.
We split the nonlinear term into
\begin{equation}\label{rbs}
\begin{split}
	&\int_0^t \nabla G_\theta (t-s) * (u\nabla(-\Delta)^{-1}u) (s) ds\\
	=&
	\sum_{|\beta| = 1} \nabla^\beta \nabla G_\theta (t) \cdot \int_0^\infty \int_{\mathbb{R}^n} (-y)^\beta (u \nabla (-\Delta)^{-1} u) (s,y) dyds
	+ r_1 (t) + r_2 (t) + r_3 (t),
\end{split}
\end{equation}
where
\[
\begin{split}
	r_1 (t)
	=&
	\int_0^{t/2} \int_{\mathbb{R}^n}
		\bigl( \nabla G_\theta (t-s,x-y) - \sum_{|\beta| = 1} \nabla^\beta \nabla G_\theta (t,x) (-y)^\beta \bigr)
		\cdot
		(u \nabla (-\Delta)^{-1} u) (s,y)
	dyds,\\
	r_2 (t)
	=&
	\int_{t/2}^t
		G_\theta (t-s)
		*
		\nabla\cdot (u \nabla (-\Delta)^{-1} u) (s)
	ds,\\
	r_3 (t)
	=&
	- \sum_{|\beta| = 1} \nabla^\beta \nabla G_\theta (t) \cdot \int_{t/2}^\infty \int_{\mathbb{R}^n} (-y)^\beta (u \nabla (-\Delta)^{-1} u) (s,y) dyds.
\end{split}
\]
Since $\int_{\mathbb{R}^n} u \nabla (-\Delta)^{-1} u dy = 0$, $r_1$ is represented by
\[
\begin{split}
	r_1 (t)
	=&
	\int_0^{t/2} \int_{\mathbb{R}^n}
		\bigl( \nabla G_\theta (t-s,x-y) - \sum_{|\beta| \le 1} \nabla^\beta \nabla G_\theta (t-s,x) (-y)^\beta \bigr)
		\cdot
		(u \nabla (-\Delta)^{-1} u) (s,y)
	dyds\\
	&+
	\sum_{|\beta|=1} \int_0^{t/2} \int_{\mathbb{R}^n}
		\bigl( \nabla^\beta \nabla G_\theta (t-s,x) - \nabla^\beta \nabla G_\theta (t,x) \bigr)
		\cdot
		(-y)^\beta (u \nabla (-\Delta)^{-1} u) (s,y)
	dyds.
\end{split}
\]
For some $R(t) = o (t^{1/\theta})~ (t\to\infty)$, we divide $r_1$ to $r_1 = r_{1,1} + r_{1,2} + r_{1,3}$, where
\[
\begin{split}
	r_{1,1} (t)
	=&
	\int_0^{t/2} \int_{|y| \le R(t)}
		\biggl( \nabla G_\theta (t-s,x-y) - \sum_{|\beta| \le 1} \nabla^\beta \nabla G_\theta (t-s,x) (-y)^\beta \biggr)
		\cdot
		(u \nabla (-\Delta)^{-1} u) (s,y)
	dyds,\\
	r_{1,2} (t)
	=&
	\int_0^{t/2} \int_{|y| > R(t)}
		\biggl( \nabla G_\theta (t-s,x-y) - \sum_{|\beta| \le 1} \nabla^\beta \nabla G_\theta (t-s,x) (-y)^\beta \biggr)
		\cdot
		(u \nabla (-\Delta)^{-1} u) (s,y)
	dyds,\\
	r_{1,3} (t)
	=&
	\sum_{|\beta| = 1} \int_0^{t/2} \int_{\mathbb{R}^n}
		\biggl( \nabla^\beta \nabla G_\theta (t-s,x) - \nabla^\beta \nabla G_\theta (t,x) \biggr) \cdot
		(-y)^\beta
		(u \nabla (-\Delta)^{-1} u) (s,y)
	dyds.
\end{split}
\]
Taylor's theorem yields that
\[
\begin{split}
	r_{1,1} (t)
	=&
	\sum_{|\beta| = 2}
	\int_0^{t/2} \int_{|y| \le R(t)} \int_0^1
		\frac{\nabla^\beta \nabla G_\theta (t-s,x-y+\lambda y)}{\beta!}
		\cdot
		\lambda (-y)^\beta (u \nabla (-\Delta)^{-1} u) (s,y)
	d\lambda dyds,\\
	r_{1,2} (t)
	=&
	\sum_{|\beta| =1}
	\int_0^{t/2} \int_{|y| > R(t)}
	\biggl(
		\int_0^1 \nabla^\beta \nabla G_\theta (t-s,x-y+\lambda y) d\lambda
		+
		\nabla^\beta \nabla G_\theta (t-s,x)
	\biggr)\\
	&\hspace{15mm}\cdot
	(-y)^\beta (u \nabla (-\Delta)^{-1} u) (s,y)
	dyds,\\
	r_{1,3} (t)
	=&
	\sum_{|\beta|=1} \int_0^{t/2} \int_{\mathbb{R}^n} \int_0^1
		\partial_t \nabla^\beta \nabla G_\theta (t-s+\lambda s,x)
		\cdot
		(-s) (-y)^\beta (u\nabla (-\Delta)^{-1} u) (s)
	d\lambda dyds.
\end{split}
\]
By Lemma \ref{decay-lin} and Propositions \ref{HLSinfty} and \ref{lem-wt}, we have that
\[
\begin{split}
	\| r_{1,1} (t) \|_{L^q (\mathbb{R}^n)}
	\le&
	C R(t) \int_0^{t/2}
		(t-s)^{-\frac{n}\theta (1-\frac1q) - \frac3\theta}
		\left\| y (u\nabla(-\Delta)^{-1}u) (s) \right\|_{L^1 (\mathbb{R}^n)}
	ds\\
	\le&
	C R(t) \int_0^{t/2}
		(t-s)^{-\frac{n}\theta (1-\frac1q) - \frac3\theta}
		(1+s)^{-\frac{n-2}\theta} \log (e+s)
	ds.
\end{split}
\]	
Thus
\[
	\| r_{1,1} (t) \|_{L^q (\mathbb{R}^n)}
	=
	o \bigl( t^{-\frac{n}\theta (1-\frac1q) - \frac2\theta} \bigr)
\]
as $t\to\infty$.
Similarly
\[
	\| r_{1,2} (t) \|_{L^q (\mathbb{R}^n)}
	\le
	C \int_0^{t/2}
		(t-s)^{-\frac{n}\theta (1-\frac1q) -\frac2\theta}
		\bigl\| y (u\nabla(-\Delta)^{-1}u) (s,y) \bigr\|_{L^1 (|y| \ge R(t))}
	ds.
\]
Hence, by Lebesgue's monotone convergence theorem together with
\[
\begin{split}
	&\int_0^{t/2}
		(t-s)^{-\frac{n}\theta (1-\frac1q) -\frac2\theta}
		\bigl\| y (u\nabla(-\Delta)^{-1}u) (s,y) \bigr\|_{L^1 (\mathbb{R}^n)}
	ds\\
	\le&
	C \int_0^{t/2}
		(t-s)^{-\frac{n}\theta (1-\frac1q) -\frac2\theta}
		(1+s)^{-\frac{n-2}\theta} \log (e+s)
	ds
	=
	O \bigl( t^{-\frac{n}\theta (1-\frac1q) - \frac2\theta} \bigr),
\end{split}
\]
we conclude that
\[
	\| r_{1,2} (t) \|_{L^q (\mathbb{R}^n)}
	=
	o \bigl( t^{-\frac{n}\theta (1-\frac1q) - \frac2\theta} \bigr)
\]
as $t\to\infty$.
Moreover
\[
\begin{split}
	\| r_{1,3} (t) \|_{L^q (\mathbb{R}^n)}
	\le&
	C \int_0^{t/2}
		(t-s)^{-\frac{n}\theta (1-\frac1q) - \frac2\theta - 1}
		(1+s)^{-\frac{n-2}\theta} \log (e+s)
	ds.
\end{split}
\]
Thus
\[
	\| r_{1,3} (t) \|_{L^q (\mathbb{R}^n)}
	=
	o \bigl( t^{-\frac{n}\theta (1-\frac1q) - \frac2\theta} \bigr)
\]
as $t\to\infty$.
Consequently
\begin{equation}\label{r_1}
	\| r_1 (t) \|_{L^q (\mathbb{R}^n)}
	=
	o \bigl( t^{-\frac{n}\theta (1-\frac1q) - \frac2\theta} \bigr)
\end{equation}
as $t\to\infty$.
The inequality \eqref{ap-bs4} leads that
\begin{equation}\label{r_2}
	\| r_2 (t) \|_{L^q (\mathbb{R}^n)}
	=
	o \bigl( t^{-\frac{n}\theta (1-\frac1q) - \frac2\theta} \bigr)
\end{equation}
as $t\to\infty$.
Propositions \ref{HLSinfty} and \ref{lem-wt} provide that
\[
	\| r_3 (t) \|_{L^q (\mathbb{R}^n)}
	\le
	C t^{-\frac{n}\theta (1-\frac1q) - \frac2\theta}
	\int_{t/2}^\infty
		s^{-\frac{n-2}\theta} \log (e+s)
	ds
\]
and
\begin{equation}\label{r_3}
	\| r_3 (t) \|_{L^q (\mathbb{R}^n)}
	=
	o \bigl( t^{-\frac{n}\theta (1-\frac1q) - \frac2\theta} \bigr)
\end{equation}
as $t\to\infty$.
Applying \eqref{r_1}, \eqref{r_2} and \eqref{r_3} to \eqref{rbs}, we complete the proof.
\hfill$\square$
\subsection{Proof of Theorem \ref{th2}}
To show Theorems \ref{th2} and \ref{th3}, we prepare the following estimates.
\begin{proposition}\label{tl2}
	Let $n \ge 2,~ 0 < \theta \le 1$ and $\sigma > 0$.
	Assume that the solution $u$ of \eqref{DD} satisfies \eqref{7} and \eqref{8}.
	Then there exist positive constants $C$ and $T$ such that
	\[
		\bigl\| (-\Delta)^{\sigma/2} \left( u(t) - M G_\theta (t) \right) \bigr\|_{L^2 (\mathbb{R}^n)}
		\le
		\left\{
	 	\begin{array}{lr}
		C t^{-\frac{n}{2\theta} - \frac\sigma\theta} (1+t)^{-\frac1\theta}
		&
		(n \ge 3~ \text{or}~ \theta < 1)\\
		C t^{-1-\sigma} (1+t)^{-1} \log (e+t)
		&
		(n = 2~ \text{and}~ \theta = 1)
		\end{array}
		\right.
	\]
	for $t \ge T$, where $M = \int_{\mathbb{R}^n} u_0 (y) dy$.
\end{proposition}
\begin{proof}
We consider only the nonlinear term of \eqref{MS}.
By Lemma \ref{decay-lin}, we see that
\[
\begin{split}
	&\biggl\| (-\Delta)^{\sigma/2} \int_0^t
		\nabla G_\theta (t-s) * (u\nabla(-\Delta)^{-1}u) (s)
	ds \biggr\|_{L^2 (\mathbb{R}^n)}\\
	\le&
	\biggl\| \int_0^{t/2}
		\nabla (-\Delta)^{\sigma/2} G_\theta (t-s) * (u\nabla(-\Delta)^{-1}u) (s)
	ds \biggr\|_{L^2 (\mathbb{R}^n)}\\
	&+
	\biggl\| \int_{t/2}^t
		G_\theta (t-s)
		*\nabla (-\Delta)^{\sigma/2} \cdot  (u\nabla(-\Delta)^{-1}u) (s)
	ds \biggr\|_{L^2 (\mathbb{R}^n)}\\
	\le&
	C \int_0^{t/2}
		(t-s)^{-\frac{n}{2\theta} - \frac{1+\sigma}\theta}
		\bigl\| (u\nabla(-\Delta)^{-1}u) (s) \bigr\|_{L^1 (\mathbb{R}^n)}
	ds\\
	&+
	C \int_{t/2}^t
		\bigl\| \nabla (-\Delta)^{\sigma/2} \cdot (u\nabla(-\Delta)^{-1}u) (s) \bigr\|_{L^2 (\mathbb{R}^n)}
	ds.
\end{split}
\]
The H\"older inequality, Proposition \ref{HLSinfty} and \eqref{7} yield that
\[
	\bigl\| (u\nabla(-\Delta)^{-1}u) (s) \bigr\|_{L^1 (\mathbb{R}^n)}
	\le
	C (1+s)^{-\frac{n-1}\theta}
\]
for $s > 0$.
From Lemma \ref{cmm}, we have that
\[
\begin{split}
	&\bigl\| \nabla (-\Delta)^{\sigma/2} \cdot (u\nabla(-\Delta)^{-1}u) (s) \bigr\|_{L^2 (\mathbb{R}^n)}\\
	\le&
	C \left( \bigl\| \nabla (-\Delta)^{\sigma/2} u \bigr\|_{L^4 (\mathbb{R}^n)} \bigl\| \nabla (-\Delta)^{-1} u \bigr\|_{L^4 (\mathbb{R}^n)}
	+\bigl\| u \bigr\|_{L^4 (\mathbb{R}^n)}
	\bigl\| \nabla^2 (-\Delta)^{\frac\sigma{2}-1} u \bigr\|_{L^4 (\mathbb{R}^n)} \right).
\end{split}
\]
Hence, by a coupling of the Sobolev inequality and Proposition \ref{lem-dr}, Proposition \ref{HLSinfty}, and \eqref{7}, we obtain that
\[
	\bigl\| \nabla (-\Delta)^{\sigma/2} \cdot (u\nabla(-\Delta)^{-1}u) (s) \bigr\|_{L^2 (\mathbb{R}^n)}
	\le
	C s^{-\frac{n}{2\theta} - \frac{n}\theta - \frac\sigma\theta}
\]
for large $s$.
Therefore we complete the proof.
\end{proof}
\begin{proposition}\label{HLSth2}
	Let $n \ge 2,~ 0 < \theta \le 1$ and $\varepsilon > 0$.
	Assume that $u_0 \in L^1 (\mathbb{R}^n, (1+|x|^2) dx) \cap L^\infty (\mathbb{R}^n)$, and the solution $u$ of \eqref{DD} satisfies \eqref{7}.
	Then there exists positive 	constant $C$ such that
	\[
		\left\| \nabla (-\Delta)^{-1} \left( u(t) - M G_\theta (t) \right) \right\|_{L^2 (\mathbb{R}^n)}
		\le
		\left\{
 		\begin{array}{lr}
		C t^{-\frac\varepsilon\theta} (1+t)^{-\frac{n}{2\theta} + \frac\varepsilon\theta}
		&
		(n \ge 3~ \text{or}~ \theta < 1)\\
		C t^{-\varepsilon} (1+t)^{-1 + \varepsilon} \log (e+t)
		&
		(n = 2~ \text{and}~ \theta = 1)
		\end{array}
		\right.
	\]
	for $t > 0$, where $M = \int_{\mathbb{R}^n} u_0 (y) dy$.
\end{proposition}
\begin{proof}
For $k = 1,\ldots,n$, we see from \eqref{MS} that
\[
\begin{split}
	&\partial_k (-\Delta)^{-1} (u - MG_\theta)\\
	=&
	\partial_k (-\Delta)^{-1} (G_\theta (t) * u_0 - M G_\theta)
	+
	\partial_k (-\Delta)^{-1} \int_0^t
		G_\theta (t-s) * \nabla \cdot (u\nabla (-\Delta)^{-1}u) (s)
	ds\\
	=&
	\int_{\mathbb{R}^n} \int_0^1
		\partial_k (-\Delta)^{-1} \nabla G_\theta (t,x-y+\lambda y)
		\cdot (-y) u_0 (y)
	d\lambda dy\\
	&+
	\int_0^t
		\partial_k (-\Delta)^{-1} \nabla G_\theta (t-s) * (u\nabla(-\Delta)^{-1}u) (s)
	ds.
\end{split}
\]
Here we used Taylor's theorem.
By Lemma \ref{decay-lin}, we see that
\[
\begin{split}
	\biggl\|
		\int_{\mathbb{R}^2} \int_0^1
		\partial_k (-\Delta)^{-1} \nabla G_\theta (t,x-y+\lambda y)
		\cdot (-y) u_0 (y)
	d\lambda dy
	\biggr\|_{L^2 (\mathbb{R}^n)}
	\le
	C t^{-\frac{n}{2\theta}} \left\| y u_0 \right\|_{L^1 (\mathbb{R}^n)}.
\end{split}
\]
Since $u_0 \in L^1 (\mathbb{R}^n, (1+|x|^2) dx) \cap L^\infty (\mathbb{R}^n) \subset L^r (\mathbb{R}^n, (1+|x|) dx)$ for $\frac1r = \frac12 + \frac{\varepsilon}{n}$, we obtain that
\[
\begin{split}
	\biggl\|
		\int_{\mathbb{R}^2} \int_0^1
		\partial_k (-\Delta)^{-1} \nabla G_\theta (t,x-y+\lambda y)
		\cdot (-y) u_0 (y)
	d\lambda dy
	\biggr\|_{L^2 (\mathbb{R}^n)}
	\le
	C t^{-\frac\varepsilon\theta} \left\| y u_0 \right\|_{L^r (\mathbb{R}^n)}.
\end{split}
\]
Thus
\[
\begin{split}
	\biggl\|
		\int_{\mathbb{R}^2} \int_0^1
		\partial_k (-\Delta)^{-1} \nabla G_\theta (t,x-y+\lambda y)
		\cdot (-y) u_0 (y)
	d\lambda dy
	\biggr\|_{L^2 (\mathbb{R}^n)}
	\le
	C t^{-\frac\varepsilon\theta} (1+t)^{-\frac{n}{2\theta}+\frac\varepsilon\theta}.
\end{split}
\]
Similarly, we obtain that
\[
\begin{split}
	&\biggl\|
		\int_0^t
			\partial_k (-\Delta)^{-1} \nabla G_\theta (t-s) * \left( u\nabla(-\Delta)^{-1}u \right) (s) ds
	\biggr\|_{L^2 (\mathbb{R}^n)}\\
	\le&
	C \int_0^{t/2}
		(t-s)^{-\frac{n}{2\theta}}
		\left\| u\nabla(-\Delta)^{-1}u (s) \right\|_{L^1 (\mathbb{R}^n)}
	ds
	+
	C \int_{t/2}^t
		(t-s)^{-\frac\varepsilon\theta}
		\left\| u\nabla(-\Delta)^{-1}u (s) \right\|_{L^r (\mathbb{R}^n)}
	ds\\
	\le&
	C \int_0^{t/2}
		(t-s)^{-\frac{n}{2\theta}}
		(1+s)^{-\frac{n-1}\theta}
	ds
	+
	C \int_{t/2}^t
		(t-s)^{-\frac\varepsilon\theta}
		(1+s)^{-\frac{n}{2\theta} - \frac{n-1}\theta + \frac\varepsilon\theta}
	ds\\
	\le&
	\left\{
 	\begin{array}{lr}
		Ct^{-\frac{n}{2\theta}}
		&
		(n \ge 3~ \text{or}~ \theta < 1)\,\\
		C t^{-1} \log (e+t)
		&
		(n = 2~ \text{and}~ \theta = 1)
	\end{array}
	\right.
\end{split}
\]
and
\[
\begin{split}
	&\biggl\|
		\int_0^t
			\partial_k (-\Delta)^{-1} \nabla G_\theta (t-s) * \left( u\nabla(-\Delta)^{-1}u \right) (s) ds
	\biggr\|_{L^2 (\mathbb{R}^n)}\\
	\le&
	C \int_0^t
		(t-s)^{-\frac\varepsilon\theta}
		\left\| u \nabla (-\Delta)^{-1} u(s) \right\|_{L^r (\mathbb{R}^n)}
	ds
	\le
	C.
\end{split}
\]
Therefore we complete the proof.
\end{proof}

Lemma \ref{HLS}, and Propositions \ref{lem-wt}, \ref{ap} and \ref{HLSth2} affirm \eqref{coef2}
when $\theta < n-1$.\\

\noindent
{\it Proof of Theorem \ref{th2}.}
We split the nonlinear term on \eqref{MS} as follows:
\begin{equation}\label{th2bs}
\begin{split}
	&\int_0^t
		\nabla G_\theta (t-s) * (u\nabla(-\Delta)^{-1}u) (s)
	ds\\
	=&
	M^2 J(t)
	+
	\int_0^t
		\nabla G_\theta (t-s) *  \left( u \nabla (-\Delta)^{-1} u - M^2 G_\theta \nabla (-\Delta)^{-1} G_\theta \right) (s)
	ds\\
	=&
	M^2 J(t)
	+
	\sum_{|\beta| = 1} \nabla^\beta \nabla G_\theta (t,x) \cdot \int_0^\infty \int_{\mathbb{R}^2}
		(-y)^\beta \left( u \nabla (-\Delta)^{-1} u - M^2 G_\theta \nabla (-\Delta)^{-1} G_\theta \right) (s,y) dyds\\
	&+ \tilde{r}_1 (t) + \tilde{r}_2 (t) + \tilde{r}_3 (t),
\end{split}
\end{equation}
where
\[
\begin{split}
	\tilde{r}_1 (t)
	=&
	\int_0^{t/2} \int_{\mathbb{R}^2}
		\bigl( \nabla G_\theta (t-s,x-y) - \sum_{|\beta|=1} \nabla G_\theta (t,x) (-y)^\beta \bigr)\\
		&\hspace{15mm}\cdot
		 \left( u \nabla (-\Delta)^{-1} u - M^2 G_\theta \nabla (-\Delta)^{-1} G_\theta \right) (s,y)
	dyds,\\
	\tilde{r}_2 (t)
	=&
	\int_{t/2}^t
		G_\theta (t-s)
		* \nabla\cdot \left( u \nabla (-\Delta)^{-1} u - M^2 G_\theta \nabla (-\Delta)^{-1} G_\theta \right) (s)
	ds,\\
	\tilde{r}_3 (t)
	=&
	-\sum_{|\beta|=1} \nabla^\beta \nabla G_\theta (t) \cdot \int_{t/2}^\infty \int_{\mathbb{R}^2}
		(-y)^\beta  \left( u \nabla (-\Delta)^{-1} u - M^2 G_\theta \nabla (-\Delta)^{-1} G_\theta \right) (s,y)
	dyds.
\end{split}
\]
Here we used the relation $\int_{\mathbb{R}^2} (u\nabla(-\Delta)^{-1}u - M^2 G_\theta \nabla (-\Delta)^{-1} G_\theta) (s,y) dy = 0$ for $\tilde{r}_1$.
For some $R(t) > 0,~ R(t) = o (t^{1/\theta})$ as $t\to\infty$, by the similar argument as in the proof of Theorem \ref{th1}, we divide $\tilde{r}_1$ into
\begin{equation}\label{th2bs3pre}
\begin{split}
	\tilde{r}_1 (t)
	=&
	\sum_{|\beta| = 2} \int_0^{t/2} \int_{|y| \le R(t)} \int_0^1
		\frac{\nabla^\beta \nabla G_\theta (t-s,x-y+\lambda y)}{\beta!}\\
		&\hspace{15mm}\cdot
		(-\lambda) (-y)^\beta \left( u \nabla (-\Delta)^{-1} u - M^2 G_\theta \nabla (-\Delta)^{-1} G_\theta \right) (s,y)
	d\lambda dyds\\
	&+
	\sum_{|\beta|=1} \int_0^{t/2} \int_{|y| > R(t)}
		\left( \int_0^1 \nabla^\beta \nabla G_\theta (t-s,x-y+\lambda y) d\lambda + \nabla^\beta \nabla G_\theta (t-s,x) \right)\\
		&\hspace{15mm}\cdot
		 (-y)^\beta \left( u \nabla (-\Delta)^{-1} u - M^2 G_\theta \nabla (-\Delta)^{-1} G_\theta \right) (s,y)
	dyds\\
	&+
	\sum_{|\beta|=1} \int_0^{t/2} \int_{\mathbb{R}^2} \int_0^1
		\partial_t \nabla^\beta \nabla G_\theta (t-s+\lambda s,x)\\
		&\hspace{15mm}\cdot 
		(-s) (-y)^\beta \left( u \nabla (-\Delta)^{-1} u - M^2 G_\theta \nabla (-\Delta)^{-1} G_\theta \right) (s,y)
	d\lambda dyds.
\end{split}
\end{equation}
Since $u\nabla(-\Delta)^{-1}u - M^2 G_\theta \nabla (-\Delta)^{-1} G_\theta = u \nabla (-\Delta)^{-1} (u-MG_\theta) + M (u-MG_\theta) \nabla (-\Delta)^{-1} G_\theta$, we see from \eqref{7}, and Propositions \ref{ap} and \ref{HLSth2} that
\[
\begin{split}
	&\bigl\| y_j \bigl( u \nabla (-\Delta)^{-1} u - M^2 G_\theta \nabla (-\Delta)^{-1} G_\theta \bigr) \bigr\|_{L^1 (\mathbb{R}^2)}\\
	\le&
	\bigl\| y_j u \bigr\|_{L^2 (\mathbb{R}^2)} \bigl\| \nabla (-\Delta)^{-1} (u - MG_\theta) \bigr\|_{L^2 (\mathbb{R}^2)}
	+
	\bigl\| u - M G_\theta \bigr\|_{L^1 (\mathbb{R}^2)} \bigl\| y_j \nabla (-\Delta)^{-1} G_\theta \bigr\|_{L^\infty (\mathbb{R}^2)}\\
	\le&
	C s^{-\frac\varepsilon\theta} (1+s)^{-\frac1\theta + \frac\varepsilon\theta} \log (e+s).
\end{split}
\]
Lemma \ref{decay-lin} together with the above inequality provides that
\[
\begin{split}
	&\biggl\|
		\sum_{|\beta| = 2} \int_0^{t/2} \int_{|y| \le R(t)} \int_0^1
			\frac{\nabla^\beta \nabla G_\theta (t-s,x-y+\lambda y)}{\beta!}\\
			&\hspace{15mm}\cdot
			(-\lambda) (-y)^\beta \left( u \nabla (-\Delta)^{-1} u - M^2 G_\theta \nabla (-\Delta)^{-1} G_\theta \right) (s,y)
		d\lambda dyds
	\biggr\|_{L^q (\mathbb{R}^2)}\\
	\le&
	C R(t) \sum_{|\beta| = 1} \int_0^{t/2}
		(t-s)^{-\frac2\theta (1-\frac1q) - \frac3\theta}
		\bigl\| (-y)^\beta \left( u \nabla (-\Delta)^{-1} u - M^2 G_\theta \nabla (-\Delta)^{-1} G_\theta \right) (s) \bigr\|_{L^1 (\mathbb{R}^2)}
	ds\\
	\le&
	C R(t) \int_0^{t/2}
		(t-s)^{-\frac2\theta (1-\frac1q) - \frac3\theta}
		s^{-\frac\varepsilon\theta} (1+s)^{-\frac1\theta + \frac\varepsilon\theta} \log (e+s)\,
	ds
	=
	o \bigl( t^{-\frac2\theta (1-\frac1q) - \frac2\theta} \bigr)
\end{split}
\]
as $t\to\infty$.
In a similar manner to above, we have that
\[
\begin{split}
	&\biggl\| \sum_{|\beta|=1} \int_0^{t/2} \int_{\mathbb{R}^2}
		\left( \int_0^1 \nabla^\beta \nabla G_\theta (t-s,x-y+\lambda y) d\lambda + \nabla^\beta \nabla G_\theta (t-s,x) \right)\\
		&\hspace{15mm}\cdot
		 (-y)^\beta \left( u \nabla (-\Delta)^{-1} u - M^2 G_\theta \nabla (-\Delta)^{-1} G_\theta \right) (s,y)
	dyds \biggr\|_{L^q (\mathbb{R}^2)}\\
	\le&
	C \sum_{|\beta| = 1} \int_0^{t/2}
		(t-s)^{-\frac2\theta (1-\frac1q) - \frac2\theta}
		\bigl\| (-y)^\beta \left( u \nabla (-\Delta)^{-1} u - M^2 G_\theta \nabla (-\Delta)^{-1} G_\theta \right) (s) \bigr\|_{L^1 (\mathbb{R}^2)}
	ds\\
	\le&
	C \int_0^{t/2}
		(t-s)^{-\frac2\theta (1-\frac1q) - \frac2\theta}
		s^{-\frac\varepsilon\theta} (1+s)^{-\frac1\theta + \frac\varepsilon\theta} \log (e+s)\,
	ds
	=
	O \bigl( t^{-\frac2\theta (1-\frac1q) - \frac2\theta} \bigr).
\end{split}
\]
Hence Lebesgue's monotone convergence theorem yields that
\[
\begin{split}
	&\biggl\| \sum_{|\beta|=1} \int_0^{t/2} \int_{|y| \ge R(t)}
		\left( \int_0^1 \nabla^\beta \nabla G_\theta (t-s,x-y+\lambda y) d\lambda + \nabla^\beta \nabla G_\theta (t-s,x) \right)\\
		&\hspace{15mm}\cdot
		 (-y)^\beta \left( u \nabla (-\Delta)^{-1} u - M^2 G_\theta \nabla (-\Delta)^{-1} G_\theta \right) (s,y)
	dyds \biggr\|_{L^q (\mathbb{R}^2)}
	=
	o \bigl( t^{-\frac2\theta (1-\frac1q) - \frac2\theta} \bigr)
\end{split}
\]
as $t\to\infty$.
Similarly we obtain that
\[
\begin{split}
	&\biggl\|
		\sum_{|\beta|=1} \int_0^{t/2} \int_{\mathbb{R}^2} \int_0^1
			\partial_t \nabla^\beta \nabla G_\theta (t-s+\lambda s,x)\\
			&\hspace{15mm}\cdot 
			(-s) (-y)^\beta \left( u \nabla (-\Delta)^{-1} u - M^2 G_\theta \nabla (-\Delta)^{-1} G_\theta \right) (s,y)
		d\lambda dyds
	\biggr\|_{L^q (\mathbb{R}^2)}\\
	\le&
	C \sum_{|\beta| = 1} \int_0^{t/2}
		(t-s)^{-\frac2\theta (1-\frac1q) - \frac2\theta - 1}
		s \bigl\| (-y)^\beta \left( u \nabla (-\Delta)^{-1} u - M^2 G_\theta \nabla (-\Delta)^{-1} G_\theta \right) (s) \bigr\|_{L^1 (\mathbb{R}^2)}
	ds\\
	\le&
	C  \int_0^{t/2}
		(t-s)^{-\frac2\theta (1-\frac1q) - \frac2\theta - 1}
		s^{1-\frac\varepsilon\theta} (1+s)^{-\frac1\theta + \frac\varepsilon\theta} \log (e+s)
	ds
	=
	o \bigl( t^{-\frac2\theta (1-\frac1q) - \frac2\theta} \bigr)
\end{split}
\]
as $t\to\infty$.
Therefore we conclude that
\begin{equation}\label{th2bs3}
	\| \tilde{r}_1 (t) \|_{L^q (\mathbb{R}^2)}
	=
	o \bigl( t^{-\frac2\theta (1-\frac1q) - \frac2\theta} \bigr)
\end{equation}
as $t\to\infty$.
For $1 \le q < \infty$, we see that
\[
\begin{split}
	\| \tilde{r}_2 (t) \|_{L^q (\mathbb{R}^2)}
	\le&
	C \int_{t/2}^t
		\biggl\{
 			\bigl\| \nabla u \bigr\|_{L^{2q} (\mathbb{R}^2)}
			\bigl\| \nabla (-\Delta)^{-1} (u-MG_\theta) \bigr\|_{L^{2q} (\mathbb{R}^2)}
			+
			\bigl\| u \bigr\|_{L^{2q} (\mathbb{R}^2)}
			\bigl\| u-MG_\theta \bigr\|_{L^{2q} (\mathbb{R}^2)}\\
			&+
			\bigl\| \nabla (u-MG_\theta) \bigr\|_{L^{2q} (\mathbb{R}^2)}
			\bigl\| \nabla (-\Delta)^{-1} G_\theta \bigr\|_{L^{2q} (\mathbb{R}^2)}
			+
			\bigl\| u-MG_\theta \bigr\|_{L^{2q} (\mathbb{R}^2)}
			\bigl\| G_\theta \bigr\|_{L^{2q} (\mathbb{R}^2)}
		\biggr\}	
	ds.
\end{split}
\]
From Proposition \ref{HLSth2}, or Proposition \ref{ap} together with Hardy-Littlewood-Sobolev's inequality leads that
\[
	\bigl\|
		\nabla (-\Delta)^{-1} (u-MG_\theta)
	\bigr\|_{L^{2q} (\mathbb{R}^2)}
	\le
	C t^{-\frac2\theta (1-\frac1{2q})} \log (e+t).
\]
We choose $\sigma = 1 - \frac1q$, then, by the Sobolev inequality and Propositions \ref{lem-dr} and \ref{tl2}, we have that
\[
\begin{split}
	\bigl\| \nabla u \bigr\|_{L^{2q} (\mathbb{R}^2)}
	\le&
	C \bigl\| \nabla (-\Delta)^{\sigma/2} u \bigr\|_{L^2 (\mathbb{R}^2)}
	\le
	C t^{-\frac2{\theta} (1-\frac1{2q}) - \frac1\theta},\\
	\bigl\| \nabla (u-M G_\theta) \bigr\|_{L^{2q} (\mathbb{R}^2)}
	\le&
	C \bigl\| \nabla (-\Delta)^{\sigma/2} (u-MG_\theta) \bigr\|_{L^2 (\mathbb{R}^2)}
	\le
	C t^{-\frac2{\theta} (1-\frac1{2q}) - \frac2\theta}.
\end{split}
\]
When $q = \infty$, we obtain that
\[
\begin{split}
	&\bigl\|
		\tilde{r}_2 (t)
	\bigr\|_{L^\infty (\mathbb{R}^2)}\\
	\le&
	C \int_{t/2}^t
		(t-s)^{-\frac2{\theta p}}
		\biggl\{
 			\bigl\| \nabla u \bigr\|_{L^{2p} (\mathbb{R}^2)}
			\bigl\| \nabla (-\Delta)^{-1} (u-MG_\theta) \bigr\|_{L^{2p} (\mathbb{R}^2)}
			+
			\bigl\| u \bigr\|_{L^{2p} (\mathbb{R}^2)}
			\bigl\| u-MG_\theta \bigr\|_{L^{2p} (\mathbb{R}^2)}\\
			&+
			\bigl\| \nabla (u-MG_\theta) \bigr\|_{L^{2p} (\mathbb{R}^2)}
			\bigl\| \nabla (-\Delta)^{-1} G_\theta \bigr\|_{L^{2p} (\mathbb{R}^2)}
			+
			\bigl\| u-MG_\theta \bigr\|_{L^{2p} (\mathbb{R}^2)}
			\bigl\| G_\theta \bigr\|_{L^{2p} (\mathbb{R}^2)}
		\biggr\}	
	ds
\end{split}
\]
for some $2/\theta < p < \infty$.
Hence we can treat $\| \tilde{r}_2 (t) \|_{L^\infty (\mathbb{R}^2)}$ in a similar manner to above.
Thus we conclude that
\begin{equation}\label{th2bs4}
\begin{split}
	\| \tilde{r}_2 (t) \|_{L^q (\mathbb{R}^2)}
	\le&
	\int_{t/2}^t
		s^{-\frac2\theta (1-\frac1q) - \frac3\theta} \log (e+s)
	ds
	=
	o \bigl( t^{-\frac2\theta (1-\frac1q) - \frac2\theta} \bigr)
\end{split}
\end{equation}
as $t \to \infty$ for $1 \le q \le \infty$.
Propositions \ref{lem-wt}, \ref{ap} and \ref{HLSth2} give that
\[
\begin{split}
	&\int_{t/2}^\infty \int_{\mathbb{R}^2}
		\left|
			y_j
			\bigl( u\nabla (-\Delta)^{-1} u - M^2 G_\theta \nabla (-\Delta)^{-1} G_\theta \bigr)
		\right|
	dyds\\
	\le&
	\int_{t/2}^\infty
		\bigl\{
		\bigl\| y_j u \bigr\|_{L^2 (\mathbb{R}^2)}
		\bigl\| \nabla (-\Delta)^{-1} (u-MG_\theta) \bigr\|_{L^2 (\mathbb{R}^2)}
		+
		M\bigl\| u-MG_\theta \bigr\|_{L^1 (\mathbb{R}^2)}
		\bigl\| y_j \nabla (-\Delta)^{-1} G_\theta \bigr\|_{L^\infty (\mathbb{R}^2)}
		\bigr\}
	ds\\
	\le&
	C \int_{t/2}^\infty
		s^{-1/\theta} \log (e+s)
	ds
\end{split}
\]
and
\begin{equation}\label{th2bs5}
	\| \tilde{r}_3 (t) \|_{L^q (\mathbb{R}^2)}
	=
	o \bigl( t^{-\frac2\theta (1-\frac1q) - \frac2\theta} \bigr)
\end{equation}
as $t \to \infty$ for $1 \le q \le \infty$.
Applying \eqref{th2bs3}--\eqref{th2bs5} to \eqref{th2bs}, we complete the proof.
\hfill$\square$

\subsection{Proof of Theorem \ref{th3}}
Lemma \ref{ap-lin} provides the estimate for the linear term on \eqref{MS}.
We divide the nonlinear term into
\[
\begin{split}
	&\int_0^t
		\nabla P (t-s) * (u\nabla (-\Delta)^{-1}u) (s)
	ds\\
	=&
	M^2 \int_0^t
		\nabla P (t-s) * (P\nabla (-\Delta)^{-1}P) (1+s)
	ds\\
	&+
	\int_0^t
		\nabla P (t-s) * \left( u\nabla(-\Delta)^{-1}u(s) - M^2 P \nabla(-\Delta)^{-1}P (1+s) \right)
	ds\\
	=&
	M^2 \tilde{J} (t)
	+
	M^2 \sum_{|\beta| = 1} \nabla^\beta \nabla P (t) \cdot \int_0^{t/2} \int_{\mathbb{R}^3}
		(-y)^\beta \left( P \nabla (-\Delta)^{-1} P \right) (1+s,y)
	dyds\\
	&+
	\sum_{|\beta| = 1} \nabla^\beta \nabla P (t) \cdot \int_0^\infty \int_{\mathbb{R}^3}
		(-y)^\beta \left( u\nabla(-\Delta)^{-1}u (s,x) - P\nabla(-\Delta)^{-1}P (1+s,y) \right)
	dyds\\
	&
	+ \varrho_1 (t) + \cdots + \varrho_5 (t),
\end{split}
\]
where
\[
\begin{split}
	\varrho_1 (t)
	=&
	\int_0^{t/2} \int_{\mathbb{R}^3}
		\biggl( \nabla P (t-s,x-y) - \sum_{|\beta| \le 1} \nabla^\beta \nabla P (t,x) (-y)^\beta \biggr)\\
		&\hspace{5mm}\cdot
		\left( u\nabla(-\Delta)^{-1}u (s,x) - M^2 P\nabla(-\Delta)^{-1}P (1+s,y) \right)
	dyds,\\
	\varrho_2 (t)
	=&
	\int_{t/2}^t P (t-s) * \nabla \cdot \left( u\nabla(-\Delta)^{-1}u (s,x) - M^2 P\nabla(-\Delta)^{-1}P (1+s,y) \right)
	dyds,\\
	\varrho_3 (t)
	=&
	- \sum_{|\beta|=1} \nabla^\beta \nabla P (t) \cdot \int_{t/2}^\infty \int_{\mathbb{R}^3}
		(-y)^\beta \left( u\nabla(-\Delta)^{-1}u (s,x) - M^2 P\nabla(-\Delta)^{-1}P (1+s,y) \right)
	dyds,\\
	\varrho_4 (t)
	=&
	M^2 \int_0^{t/2} \int_{\mathbb{R}^3}
		\biggl( \nabla P (t-s,x-y) - \sum_{|\beta| \le 1} \nabla^\beta \nabla P (t,x) (-y)^\beta \biggr)\\
		&\hspace{5mm}\cdot
		\left( P \nabla (-\Delta)^{-1} P (1+s,y) - P \nabla (-\Delta)^{-1} P (s,y) \right)
	dyds,\\
	\varrho_5 (t)
	=&
	M^2 \int_{t/2}^t
		P(t-s) * \nabla \cdot \left( P \nabla (-\Delta)^{-1} P (1+s,y) - P \nabla (-\Delta)^{-1} P (s,y) \right)
	dyds\\
	&- M^2 \sum_{|\beta|=1} \nabla^\beta \nabla P (t) \cdot \int_{t/2}^t \int_{\mathbb{R}^3}
		(-y)^\beta \left( P \nabla (-\Delta)^{-1} P (1+s,y) - P \nabla (-\Delta)^{-1} P (s,y) \right)
	dyds.
%
\end{split}
\]
We note that
\[
\begin{split}
	&\sum_{|\beta| = 1} \nabla^\beta \nabla P (t) \cdot \int_0^{t/2} \int_{\mathbb{R}^3}
		(-y)^\beta \left( P \nabla (-\Delta)^{-1} P \right) (1+s,y)
	dyds\\
	=&
	-\frac13 \Delta P (t) \int_0^{t/2} (1+s)^{-1} ds \int_{\mathbb{R}^3} y \cdot (P\nabla(-\Delta)^{-1}P) (1,y) dy
	= \tilde{K} (t),
\end{split}
\]
since $P \partial_j (-\Delta)^{-1} P$ is an odd function in $x_j$.
The same argument as in the proof of Theorem \ref{th2} leads that
\[
	\bigl\| \varrho_1 (t) \bigr\|_{L^q (\mathbb{R}^3)}
	+
	\bigl\| \varrho_2 (t) \bigr\|_{L^q (\mathbb{R}^3)}
	+
	\bigl\| \varrho_3 (t) \bigr\|_{L^q (\mathbb{R}^3)}
	=
	o \bigl( t^{-3 (1-\frac1q) - 2} \bigr)
\]
as $t\to\infty$ for $1 \le q \le \infty$.
Lemma \ref{decay-lin} together with Taylor's theorem describes that
\[
\begin{split}
	\varrho_4 (t)
	=&
	\sum_{|\beta| = 2} \int_0^{t/2} \int_{\mathbb{R}^3} \int_0^1 \int_0^1
		\frac{\nabla^\beta \nabla P (t-s,x-y + \lambda y)}{\beta!} \lambda (-y)^\beta\\
		&\hspace{5mm}\cdot
		\partial_t \left( P \nabla (-\Delta)^{-1} P \right) (s+\mu,y)
	d\mu d\lambda dyds\\
	&+
	\sum_{|\beta|=1} \int_0^{t/2} \int_{\mathbb{R}^3} \int_0^1 \int_0^1
		\partial_t \nabla^\beta \nabla P (t-s+\lambda s,x) (-s) (-y)^\beta\\
		&\hspace{5mm}\cdot
		\partial_t \left( P \nabla (-\Delta)^{-1} P \right) (s+\mu,y)
	d\mu d\lambda dyds
\end{split}
\]
and
\[
	\bigl\| \varrho_4 (t) \bigr\|_{L^q (\mathbb{R}^3)}
	\le
	C \int_0^{t/2} \int_0^1
		(t-s)^{-3 (1-\frac1q) - 3}
		(s+\mu)^{-1}
	d\mu ds.
\]
Similarly we obtain that
\[
\begin{split}
	\bigl\| \varrho_5 (t) \bigr\|_{L^q (\mathbb{R}^3)}
	\le&
	C \int_{t/2}^t \int_0^1
		(s+\mu)^{-3(1-\frac1q) -4}
	d\mu ds
	+
	C t^{-3(1-\frac1q) - 2} \int_{t/2}^t \int_0^1
		(s+\mu)^{-2}
	d\mu ds.
\end{split}
\]
Therefore $\varrho_4$ and $\varrho_5$ fulfill that
\[
	\bigl\| \varrho_4 (t) \bigr\|_{L^q (\mathbb{R}^3)}
	+
	\bigl\| \varrho_5 (t) \bigr\|_{L^q (\mathbb{R}^3)}
	=
	o \bigl( t^{-3(1-\frac1q)-2} \bigr)
\]
as $t\to\infty$ for $1 \le q \le \infty$.
Therefore we derive the assertion.
\hfill$\square$
\subsection{Proof of Theorem \ref{th4}}
Before proving Theorem \ref{th4} we prepare the following proposition.
\begin{proposition}\label{tl4}
Upon the assumption of Theorem \ref{th4},
\begin{equation}\label{tl4-1}
\begin{split}
	&\bigl\|
		u(t) - MP(t) - m\cdot \nabla P(1+t) - M^2 J(1+t)
	\bigr\|_{L^p (\mathbb{R}^2)}\\
	\le&
	Ct^{-2(1-\frac1p)} (1+t)^{-2} \left( \log (e+t) \right)^2
\end{split}
\end{equation}
for $t > 0$ and $1 \le p \le \infty$.
Moreover, for $\sigma > 0$, there exist positive constants $C$ and $T$ such that
\begin{equation}\label{tl4-2}
\begin{split}
	&\bigl\|
		(-\Delta)^{\sigma/2} \left( u(t) - M P (t) - m \cdot \nabla P(1+t) - M^2 J (1+t) \right)
	\bigr\|_{L^2 (\mathbb{R}^2)}\\
	\le&
	C t^{-1 - \sigma} (1+t)^{-2} \left( \log (e+t) \right)^2
\end{split}
\end{equation}
for $t \ge T$.
\end{proposition}
\begin{proof}
We show \eqref{tl4-2}.
From \eqref{MS} we see that
\[
\begin{split}
	&u(t) - M P (t) - m \cdot \nabla P(1+t) - M^2 J (1+t)\\
	=&
	P(t) * u_0 - MP (t) - m\cdot\nabla P (1+t)
	+
	\int_0^t P(t-s) * \nabla \cdot (u\nabla(-\Delta)^{-1}u) (s) ds
	-
	M^2 J(1+t).
\end{split}
\]
Since
\[
\begin{split}
	&P(t) * u_0 - MP (t) - m\cdot\nabla P (1+t)\\
	=&
	P(t) * u_0 - MP (t) - m\cdot\nabla P (t)
	+
	m \cdot \nabla \left( P(t) - P (1+t) \right)\\
	=&
	\sum_{|\alpha| = 2} \int_{\mathbb{R}^2} \int_0^1 \frac{\nabla^\alpha P (t,x-y+\lambda y)}{\alpha!} \lambda (-y)^\alpha u_0 (y) d\lambda dy
	-
	m \cdot \int_0^1 \partial_t \nabla P(t+\mu) d\mu,
\end{split}
\]
we have for $\sigma > 0$ that
\[
\begin{split}
	\left\| (-\Delta)^{\sigma/2} \left( P(t) * u_0 - MP (t) - m\cdot\nabla P (1+t) \right) \right\|_{L^2 (\mathbb{R}^2)}
	\le
	C t^{-3 - \sigma}.
\end{split}
\]
From \eqref{J}, we obtain that
\[
\begin{split}
	&(-\Delta)^{\sigma/2} \biggl( \int_0^t
		P(t-s) * \nabla \cdot (u\nabla(-\Delta)^{-1}u) (s)
	ds
	-
	M^2 J(1+t) \biggr)\\
	=&
	\int_0^{t/2}
		\nabla (-\Delta)^{\sigma/2} P(t-s)
		*
		\left( u\nabla(-\Delta)^{-1}u - M^2 P\nabla(-\Delta)^{-1}P \right) (s)
	ds\\
	&+
	\int_{t/2}^t
		P(t-s)
		*
		\nabla (-\Delta)^{\sigma/2} \cdot \left( u\nabla(-\Delta)^{-1}u - M^2 P\nabla(-\Delta)^{-1}P \right) (s)
	ds\\
	&+ M^2 (-\Delta)^{\sigma/2} \left( J(t) - J(1+t) \right).
\end{split}
\]
Taylor's theorem together with the relation $\int_{\mathbb{R}^2} (u\nabla(-\Delta)^{-1}u - M^2 P\nabla(-\Delta)^{-1}P) dy = 0$ gives that
\[
\begin{split}
	&\int_0^{t/2}
		\nabla (-\Delta)^{\sigma/2} P(t-s)
		*
		\left( u\nabla(-\Delta)^{-1}u - M^2 P\nabla(-\Delta)^{-1}P \right) (s)
	ds\\
	=&
	\sum_{|\beta|=1} \int_0^{t/2} \int_{\mathbb{R}^2} \int_0^1
		\nabla^\beta \nabla (-\Delta)^{\sigma/2} P(t-s,x-y + \lambda y)\\
		&\hspace{15mm}
		\cdot
		(-y)^\beta \left( u\nabla(-\Delta)^{-1}u - M^2 P\nabla(-\Delta)^{-1}P \right) (s,y)
	d\lambda dyds.
\end{split}
\]
Hence
\[
\begin{split}
	&\biggl\|
		\int_0^{t/2}
			\nabla (-\Delta)^{\sigma/2} P(t-s)
			*
			\left( u\nabla(-\Delta)^{-1}u - M^2 P\nabla(-\Delta)^{-1}P \right) (s)
		ds
	\biggr\|_{L^2 (\mathbb{R}^2)}\\
	\le&
	C \sum_{|\beta|=1} \int_0^{t/2}
		(t-s)^{-3 - \sigma}
		\bigl\| (-y)^\beta \left( u\nabla(-\Delta)^{-1}u - M^2 P\nabla(-\Delta)^{-1}P \right) (s) \bigr\|_{L^1 (\mathbb{R}^2)}
	ds.
\end{split}
\]
Propositions \ref{lem-wt} and \ref{HLSth2} lead that
\[
\begin{split}
	&\bigl\| (-y)^\beta \left( u\nabla(-\Delta)^{-1}u - M^2 P\nabla(-\Delta)^{-1}P \right) \bigr\|_{L^1 (\mathbb{R}^2)}\\
	\le&
	\bigl\| (-y)^\beta u \bigr\|_{L^2 (\mathbb{R}^2)}
	\bigl\| \nabla(-\Delta)^{-1} (u - MP) \bigr\|_{L^2 (\mathbb{R}^2)}
	+
	M \bigl\| u-MP \bigr\|_{L^1 (\mathbb{R}^2)}
	\bigl\| (-y)^\beta \nabla (-\Delta)^{-1} P \bigr\|_{L^\infty (\mathbb{R}^2)}\\
	\le&
	C s^{-\varepsilon} (1+s)^{-1+\varepsilon} \log (e+s).
\end{split}
\]
Thus
\[
\begin{split}
	&\biggl\|
		\int_0^{t/2}
			\nabla (-\Delta)^{\sigma/2} P(t-s)
			*
			\left( u\nabla(-\Delta)^{-1}u - M^2 P\nabla(-\Delta)^{-1}P \right) (s)
		ds
	\biggr\|_{L^2 (\mathbb{R}^2)}
	\le
	C t^{-3 - \sigma} \left( \log (e+t) \right)^2.
\end{split}
\]
From Lemma \ref{cmm}, we have that
\[
\begin{split}
	&\biggl\|
		\int_{t/2}^t
		P(t-s)
		*
		\nabla (-\Delta)^{\sigma/2}\cdot \left( u\nabla(-\Delta)^{-1}u - M^2 P\nabla(-\Delta)^{-1}P \right) (s)
	ds
	\biggr\|_{L^2 (\mathbb{R}^2)}\\
	\le&
	C \int_{t/2}^t
		\bigl\{ \bigl\| \nabla (-\Delta)^{\sigma/2} u \bigr\|_{L^4 (\mathbb{R}^2)}
		\bigl\| \nabla (-\Delta)^{-1} (u-MP) \bigr\|_{L^4 (\mathbb{R}^2)}
		+
		\bigl\| u \bigr\|_{L^4 (\mathbb{R}^n)}
		\bigl\| \nabla^2 (-\Delta)^{\frac\sigma{2}-1} (u-MP) \bigr\|_{L^4 (\mathbb{R}^2)}\\
		&+
		\bigl\| \nabla (-\Delta)^{\sigma/2} (u-MP) \bigr\|_{L^4 (\mathbb{R}^2)}
		\bigl\| \nabla (-\Delta)^{-1} P \bigr\|_{L^4 (\mathbb{R}^2)}
		+
		\bigl\| u-MP \bigr\|_{L^4 (\mathbb{R}^2)}
		\bigl\| \nabla^2 (-\Delta)^{\frac\sigma{2}-1} P \bigr\|_{L^4 (\mathbb{R}^2)} \bigr\}
	ds.
\end{split}
\]
Therefore, by Propositions \ref{lem-dr} and \ref{tl2} with the aid of the Sobolev inequality, and Proposition \ref{ap} and \eqref{7}, we obtain that
\[
\begin{split}
	&\biggl\|
		\int_{t/2}^t
		P(t-s)
		*
		(-\Delta)^{\sigma/2}\nabla\cdot \left( u\nabla(-\Delta)^{-1}u - M^2 P\nabla(-\Delta)^{-1}P \right) (s)
	ds
	\biggr\|_{L^2 (\mathbb{R}^2)}\\
	\le&
	C \int_{t/2}^t
		s^{-4-\sigma} \log (e+s)
	ds
\end{split}
\]
for large $t$.
Since
\[
\begin{split}
	(-\Delta)^{\sigma/2} \left( J(1+t) - J(t) \right)
	=
	\int_0^1 \partial_t (-\Delta)^{\sigma/2} J(t+\mu) d\mu
\end{split}
\]
and
\[
\begin{split}
	&\partial_t (-\Delta)^{\sigma/2} J(t)
	=
	\nabla (-\Delta)^{\sigma/2} \cdot ( P \nabla (-\Delta)^{-1} P) (t)\\
	&-
	\sum_{|\beta|=1} \int_0^{t/2} \int_{\mathbb{R}^2} \int_0^1
		\nabla^\beta \nabla (-\Delta)^{(1+\sigma)/2} P(t-s, x-y+\lambda y) \cdot (-y)^\beta (P\nabla(-\Delta)^{-1}P) (s,y)
	d\lambda dyds\\
	&-
	\int_{t/2}^t
		P(t-s) * \nabla (-\Delta)^{(1+\sigma)/2} \cdot (P\nabla(-\Delta)^{-1}P) (s)
	ds,
\end{split}
\]
we see that
\[
	\bigl\| (-\Delta)^{\sigma/2} (J(1+t) - J(t)) \bigr\|_{L^2 (\mathbb{R}^2)}
	\le
	C \int_0^1 (t+\mu)^{-3-\sigma} d\mu.
\]
Consequently, we obtain \eqref{tl4-2}.
The Minkowski inequality and \eqref{7} lead \eqref{tl4-1} for small $t$.
For large $t$, \eqref{tl4-1} is derived in a similar manner to above.
\end{proof}
We remark that the proof for \eqref{tl4-1} does not require Lemma \ref{cmm}.
Thus we can show \eqref{tl4-1} even for $p = 1$.
\begin{proposition}\label{HLSth4}
Upon the assumption of Theorem \ref{th4},
\[
\begin{split}
	&\left\| \nabla (-\Delta)^{-1} \left( u(t) - M P (t) - m\cdot\nabla P (1+t) - M^2 J (1+t) \right) \right\|_{L^2 (\mathbb{R}^2)}
	\le
	C (1+t)^{-2} \left( 1 + \left| \log t \right| \right)
\end{split}
\]
for $t > 0$.
\end{proposition}
\begin{proof}
From \eqref{MS}, we see that
\[
\begin{split}
	&\nabla (-\Delta)^{-1} \left( u(t) - M P (t) - m\cdot\nabla P (1+t) - M^2 J (1+t) \right)\\
	=&
	\nabla (-\Delta)^{-1}  \left( P(t) * u_0 - M P (t) - m\cdot\nabla P (1+t) \right)\\
	&+
	\nabla (-\Delta)^{-1} \biggl( \int_0^t P (t-s) * \nabla\cdot(u\nabla (-\Delta)^{-1} u) (s) ds - M^2 J (1+t) \biggr).
\end{split}
\]
We estimate the first part.
Since
\[
	P(t) * u_0 - M P(t) - m\cdot\nabla P(1+t)
	=
	P(t) * u_0 - M P(t) - m\cdot\nabla P(t) + m \cdot \nabla \left( P(t) - P(1+t) \right),
\]
we have that
\[
\begin{split}
	&\nabla (-\Delta)^{-1}  \left( P(t) * u_0 - M P (t) - m\cdot\nabla P (1+t) \right)\\
	=&
	\sum_{|\alpha| = 2} \int_{\mathbb{R}^2} \int_0^1
		\frac{\nabla (-\Delta)^{-1} \nabla^\alpha P (t,x-y+\lambda y)}{\alpha!} (-\lambda) (-y)^\alpha u_0 (y)
	d\lambda dy\\
	&+
	\int_0^1 \nabla (-\Delta)^{-1} (m \cdot \nabla) \partial_t P (t+\mu) d\mu
\end{split}
\]
from Taylor's theorem.
Hence Lemma \ref{decay-lin} yields that
\[
	\left\| \nabla (-\Delta)^{-1} \left( P(t) * u_0 - M P(t) - m\cdot\nabla P(1+t) \right) \right\|_{L^2 (\mathbb{R}^2)}
	\le
	C t^{-2}.
\]
On the other hand, from
\[
\begin{split}
	&P(t) * u_0 - M P(t) - m\cdot\nabla P(1+t)\\
	=&
	- \int_0^1 \partial_t P (t+\lambda) * u_0  d\lambda
	+
	\int_{\mathbb{R}^2} \int_0^1 \nabla P(1+t,x-y+\lambda y) \cdot (-y) u_0 (y) d\lambda dy
	-
	M \int_0^1 \partial_t P(t+\lambda) d\lambda,
\end{split}
\]
this part fulfills that
\[
\begin{split}
	&\left\| \nabla (-\Delta)^{-1} \left( P(t) * u_0 - M P(t) - m\cdot\nabla P(1+t) \right) \right\|_{L^2 (\mathbb{R}^2)}\\
	\le&
	C \int_0^1 (t+\lambda)^{-1} d\lambda
	\le
	C \log (1+\tfrac1{t}).
\end{split}
\]
Therefore we obtain that
\[
	\left\| \nabla (-\Delta)^{-1} \left( P(t) * u_0 - M P(t) - m\cdot\nabla P(1+t) \right) \right\|_{L^2 (\mathbb{R}^2)}
	\le
	C (1+t)^{-2} \left( 1 + \left| \log t \right| \right).
\]
For the nonlinear term, we have that
\[
\begin{split}
	&\nabla (-\Delta)^{-1} \biggl( \int_0^t P (t-s) * \nabla \cdot (u\nabla (-\Delta)^{-1} u) (s) ds - M^2 J (1+t) \biggr)\\
	=&
	\nabla^2 (-\Delta)^{-1} \int_0^{t/2} \int_{\mathbb{R}^2}
		\left( P(t-s,x-y) - P (t-s,x) \right) \cdot \left( u \nabla (-\Delta)^{-1} u - M^2 P \nabla (-\Delta)^{-1} P \right) (s,y)
	dyds\\
	&+
	\nabla^2 (-\Delta)^{-1} \int_{t/2}^t
		P(t-s) * \left( u \nabla (-\Delta)^{-1} u - M^2 P \nabla (-\Delta)^{-1} P \right) (s)
	ds\\
	&+
	M^2 \nabla (-\Delta)^{-1} \left( J(t) - J(1+t) \right).
\end{split}
\]
By Lemma \ref{decay-lin} with Taylor's theorem, we obtain that
\[
\begin{split}
	&\biggl\| \nabla (-\Delta)^{-1} \biggl( \nabla \cdot \int_0^t P (t-s) * (u\nabla (-\Delta)^{-1} u) (s) ds - M^2 J (1+t) \biggr) \biggr\|_{L^2 (\mathbb{R}^2)}\\
	\le&
	C \sum_{|\beta| = 1} \int_0^{t/2}
		(t-s)^{-2} \bigl\| (-y)^\beta \left( u \nabla (-\Delta)^{-1} u - M^2 P \nabla (-\Delta)^{-1} P \right) (s) \bigr\|_{L^1 (\mathbb{R}^2)}
	ds\\
	&+
	C \int_{t/2}^t
		(t-s)^{-1/3}
		\left\| \left( u \nabla (-\Delta)^{-1} u - M^2 P \nabla (-\Delta)^{-1} P \right) (s) \right\|_{L^{3/2} (\mathbb{R}^2)}
	ds
	+ Ct^{-2}\\
	\le&
	C \int_0^{t/2}
		(t-s)^{-2} s^{-\varepsilon} (1+s)^{-1+\varepsilon} \log (e+s)
	ds
	+
	C \int_{t/2}^t
		(t-s)^{-1/3} s^{-8/3}
	ds
	+
	C t^{-2}\\
	\le&
	C t^{-2} \log (e+t).
\end{split}
\]
Therefore we complete the proof.
\end{proof}
Since
\[
\begin{split}
	&u\nabla(-\Delta)^{-1}u -M^2 P\nabla(-\Delta)^{-1}P
	-
	M \left( P \nabla(-\Delta)^{-1} (m\cdot\nabla P + M^2 J) + (m\cdot\nabla P + M^2 J) \nabla (-\Delta)^{-1} P \right)\\
	=&
	u \nabla (-\Delta)^{-1} \left( u - M P - m\cdot\nabla P - M^2 J \right)
	+
	M \left( u - M P - m\cdot\nabla P - M^2 J \right) \nabla (-\Delta)^{-1} P\\
	&+
	\left( u - M P \right) \nabla (-\Delta)^{-1} (m\cdot\nabla) P
	+
	M^2 \left( u - M P \right) \nabla (-\Delta)^{-1} J,
\end{split}
\]
we see \eqref{coef4} from Propositions \ref{tl4} and \ref{HLSth4}.\\

\noindent
{\it Proof of Theorem \ref{th4}.}
The decay of the first term on the right hand side of \eqref{MS} is treated by Lemma \ref{ap-lin}.
We divide the second term as
\[
\begin{split}
	&\int_0^t \nabla P(t-s) * (u\nabla(-\Delta)^{-1}u) (s) ds\\
	=&
	M^2 J(t) + \int_0^t \nabla P(t-s) * \bigl\{
		u\nabla(-\Delta)^{-1}u(s) -M^2 P\nabla(-\Delta)^{-1}P (s)\\
		&\hspace{20mm}-
		M \left( P \nabla(-\Delta)^{-1} (m\cdot\nabla) P
				+ (m\cdot\nabla) P \nabla(-\Delta)^{-1} P \right) (1+s)\\
		&\hspace{20mm}-
		M^3 \left( P\nabla(-\Delta)^{-1}J + J \nabla (-\Delta)^{-1} P \right) (1+s)
	\bigr\}ds\\
	&+
	M \int_0^t \nabla P(t-s)
	* \left( P\nabla(-\Delta)^{-1}(m\cdot\nabla) P + (m\cdot\nabla)P \nabla (-\Delta)^{-1} P \right) (1+s) ds\\
	&+
	M^3
	\int_0^t
		\nabla P(t-s)
		*
		\left( P \nabla (-\Delta)^{-1} J + J \nabla (-\Delta)^{-1} P \right) (1+s)
	ds.
\end{split}
\]
Since $\int_{\mathbb{R}^2} u\nabla(-\Delta)^{-1}u dy = \int_{\mathbb{R}^2} P \nabla (-\Delta)^{-1} P dy = \int_{\mathbb{R}^2} ( P \nabla (-\Delta)^{-1} (m\cdot\nabla) P + (m\cdot\nabla) P \nabla (-\Delta)^{-1} P ) dy\\ = \int_{\mathbb{R}^2} (P \nabla (-\Delta)^{-1} J + J \nabla (-\Delta)^{-1} P) dy = 0$, we see that
\[
\begin{split}
	&\int_0^t \nabla P(t-s) * \bigl\{
		u\nabla(-\Delta)^{-1}u(s) -M^2 P\nabla(-\Delta)^{-1}P (s)\\
		&\hspace{15mm}-
		M \left( P \nabla(-\Delta)^{-1} (m\cdot\nabla) P
				+ (m\cdot\nabla) P \nabla(-\Delta)^{-1} P \right) (1+s)\\
		&\hspace{15mm}-
		M^3 \left( P\nabla(-\Delta)^{-1}J + J \nabla (-\Delta)^{-1} P \right) (1+s)
	\bigr\}ds\\
	=&
	\sum_{|\beta| = 1}
	\nabla^\beta \nabla P (t) \cdot
	\int_0^\infty \int_{\mathbb{R}^2}
	(-y)^\beta \bigl\{
		u\nabla(-\Delta)^{-1}u(s,y) -M^2 P\nabla(-\Delta)^{-1}P (s,y)\\
		&\hspace{15mm}-
		M \left( P \nabla(-\Delta)^{-1} (m\cdot\nabla) P
				+ (m\cdot\nabla) P \nabla(-\Delta)^{-1} P \right) (1+s,y)\\
		&\hspace{15mm}-
		M^3 \left( P\nabla(-\Delta)^{-1}J + J \nabla (-\Delta)^{-1} P \right) (1+s,y)
	\bigr\} dyds
	+
	\rho_1 (t) + \rho_2 (t) + \rho_3 (t),
\end{split}
\]
where
\[
\begin{split}
	\rho_1 (t)
	&=
	\int_0^{t/2} \int_{\mathbb{R}^2}
		\biggl(
			\nabla P(t-s,x-y)
			-
			\sum_{|\beta| \le 1} \nabla^\beta \nabla P (t,x) (-y)^\beta
		\biggr)\\
		&\hspace{10mm}\cdot
		\bigl\{
		u\nabla(-\Delta)^{-1}u(s,y) -M^2 P\nabla(-\Delta)^{-1}P (s,y)\\
		&\hspace{15mm}-
		M \left( P \nabla(-\Delta)^{-1} (m\cdot\nabla) P
				+ (m\cdot\nabla) P \nabla(-\Delta)^{-1} P \right) (1+s,y)\\
		&\hspace{15mm}-
		M^3 \left( P\nabla(-\Delta)^{-1}J + J \nabla (-\Delta)^{-1} P \right) (1+s,y)
	\bigr\} dyds,\\
	\rho_2 (t)
	&=
	\int_{t/2}^t P (t-s) * \nabla\cdot \bigl\{
		u\nabla(-\Delta)^{-1}u(s) -M^2 P\nabla(-\Delta)^{-1}P (s)\\
		&\hspace{15mm}-
		M \left( P \nabla(-\Delta)^{-1} (m\cdot\nabla) P
				+ (m\cdot\nabla) P \nabla(-\Delta)^{-1} P \right) (1+s)\\
		&\hspace{15mm}-
		M^3 \left( P\nabla(-\Delta)^{-1}J + J \nabla (-\Delta)^{-1} P \right) (1+s)
	\bigr\}ds,\\
	%
	\rho_3 (t)
	&=
	-\sum_{|\beta| = 1}
	\nabla^\beta \nabla P (t) \cdot
	\int_{t/2}^\infty \int_{\mathbb{R}^2}
	(-y)^\beta
	\bigl\{
		u\nabla(-\Delta)^{-1}u(s,y) -M^2 P\nabla(-\Delta)^{-1}P (s,y)\\
		&\hspace{15mm}-
		M \left( P \nabla(-\Delta)^{-1} (m\cdot\nabla) P
				+ (m\cdot\nabla) P \nabla(-\Delta)^{-1} P \right) (1+s,y)\\
		&\hspace{15mm}-
		M^3 \left( P\nabla(-\Delta)^{-1}J + J \nabla (-\Delta)^{-1} P \right) (1+s,y)
	\bigr\} dyds.
\end{split}
\]
Moreover, from $\int_{\mathbb{R}^2} (-y)^\beta (P\nabla(-\Delta)^{-1} (m\cdot\nabla) P + (m\cdot\nabla)P \nabla (-\Delta)^{-1}P) dy = 0$ for $|\beta| \le 1$, we obtain that
\[
\begin{split}
	&\int_0^t \nabla P(t-s)
	* \left( P\nabla(-\Delta)^{-1}(m\cdot\nabla) P + (m\cdot\nabla)P \nabla (-\Delta)^{-1} P \right) (1+s) ds\\
	=&
	\int_0^t \nabla P(t-s)
	* \left( P\nabla(-\Delta)^{-1}(m\cdot\nabla) P + (m\cdot\nabla)P \nabla (-\Delta)^{-1} P \right) (s) ds
	+\rho_4 (t) + \rho_5 (t),
\end{split}
\]
where
\[
\begin{split}
	\rho_4 (t)
	=&
	\sum_{|\beta| = 1} \int_0^{t/2} \int_{\mathbb{R}^2}
		\biggl( \nabla P (t-s,x-y) - \sum_{|\beta| \le 1} \nabla^\beta \nabla P (t-s,x) (-y)^\beta \biggr)\\
		&\hspace{5mm}\cdot \bigl\{ \bigl( P\nabla(-\Delta)^{-1} (m\cdot\nabla) P + (m\cdot\nabla)P \nabla (-\Delta)^{-1}P \bigr) (1+s,y)\\
		&\hspace{10mm} - \bigl( P\nabla(-\Delta)^{-1} (m\cdot\nabla) P + (m\cdot\nabla)P \nabla (-\Delta)^{-1}P \bigr) (s,y) \bigr\}
	dyds,\\
	\rho_5 (t)
	=&
	\int_{t/2}^t \int_0^1
		P (t-s)
		* \nabla \cdot \bigl\{ \bigl( P\nabla(-\Delta)^{-1} (m\cdot\nabla) P + (m\cdot\nabla)P \nabla (-\Delta)^{-1}P \bigr) (1+s)\\
		&\hspace{10mm} - \bigl( P\nabla(-\Delta)^{-1} (m\cdot\nabla) P + (m\cdot\nabla)P \nabla (-\Delta)^{-1}P \bigr) (s) \bigr\}
	ds.
\end{split}
\]
Similarly we have that
\[
\begin{split}
	&\int_0^t
		\nabla P(t-s)
		*
		\left( P \nabla (-\Delta)^{-1} J + J \nabla (-\Delta)^{-1} P \right) (1+s)
	ds\\
	=&
	\sum_{|\beta| = 1} \nabla^\beta \nabla P(t,x)
	\cdot \int_0^{t/2} \int_{\mathbb{R}^2}
		(-y)^\beta \left( P \nabla (-\Delta)^{-1} J + J \nabla (-\Delta)^{-1} P \right) (1+s,y)
	dyds\\
	&+
	\int_0^{t/2} \int_{\mathbb{R}^2}
		\biggl( \nabla P(t-s,x-y) - \sum_{|\beta| \le 1} \nabla^\beta \nabla P(t,x) (-y)^\beta \biggr)\\
		&\hspace{10mm}\cdot
		\left( P \nabla (-\Delta)^{-1} J + J \nabla (-\Delta)^{-1} P \right) (1+s,y)
	dyds\\
	&+
	\int_{t/2}^t
		P(t-s) * \nabla\cdot \left( P \nabla (-\Delta)^{-1} J + J \nabla (-\Delta)^{-1} P \right) (1+s)
	ds\\
	=&
	\sum_{|\beta|=1} \nabla^\beta \nabla P(t,x)
	\cdot
	\int_0^{t/2} (1+s)^{-1} ds
	\int_{\mathbb{R}^2}
		(-y)^\beta \left( P\nabla(-\Delta)^{-1}J + J\nabla(-\Delta)^{-1}P \right) (1,y)
	dy\\
	&+
	\int_0^{t/2} \int_{\mathbb{R}^2}
		\biggl( \nabla P(t-s,x-y) + (y\cdot\nabla) \nabla P(t,x) \biggr)\\
		&\hspace{15mm}\cdot
		\left( P \nabla (-\Delta)^{-1} J + J \nabla (-\Delta)^{-1} P \right) (s,y)
	dyds\\
	&+
	\int_{t/2}^t
		P(t-s) * \nabla \cdot \left( P \nabla (-\Delta)^{-1} J + J \nabla (-\Delta)^{-1} P \right) (s)
	ds
	+ \rho_6 (t) + \rho_7 (t),
\end{split}
\]
where
\[
\begin{split}
	\rho_6 (t)
	=&
	\int_0^{t/2} \int_{\mathbb{R}^2}
		\biggl( \nabla P (t-s,x-y) - \sum_{|\beta| \le 1} \nabla^\beta \nabla P(t,x) (-y)^\beta \biggr)\\
		&\hspace{5mm}\cdot
		\left( \left( P \nabla (-\Delta)^{-1} J + J \nabla (-\Delta)^{-1} P \right) (1+s,y)
		-
		\left( P \nabla (-\Delta)^{-1} J + J \nabla (-\Delta)^{-1} P \right) (s,y) \right)
	dyds,\\
	\rho_7 (t)
	=&
	\int_{t/2}^t
		P(t-s)\\
		&\hspace{5mm} *
		\nabla \cdot \left( \left( P \nabla (-\Delta)^{-1} J + J \nabla (-\Delta)^{-1} P \right) (1+s)
		-
		\left( P \nabla (-\Delta)^{-1} J + J \nabla (-\Delta)^{-1} P \right) (s) \right)
	ds.
\end{split}
\]
Now we remark that
\[
	\sum_{|\beta|=1} \nabla^\beta \nabla P(t,x)
	\cdot
	\int_0^{t/2} (1+s)^{-1} ds
	\int_{\mathbb{R}^2}
		(-y)^\beta \left( P\nabla(-\Delta)^{-1}J + J\nabla(-\Delta)^{-1}P \right) (1,y)
	dy
	=
	K(t)
\]
since $P \partial_j (-\Delta)^{-1} J + \partial_j (-\Delta)^{-1} P$ is an odd function in $x_j$ and is an even function in another spatial variable.
Consequently, we see that
\begin{equation}\label{th4bs}
\begin{split}
	&\int_0^t
		\nabla P(t-s) * (u\nabla(-\Delta)^{-1}u) (s)
	ds
	=
	M^2 J(t) + M^3 K(t) + J_2 (t)\\
	+&
	\sum_{|\beta|=1} \nabla^\beta \nabla P(t)
	\cdot
	\int_0^\infty \int_{\mathbb{R}^2}
		(-y)^\beta \bigl\{
			u\nabla (-\Delta)^{-1}u (s,y) - M^2 P \nabla (-\Delta)^{-1} P(s,y)\\
			&\hspace{10mm} -M \bigl(
				P\nabla(-\Delta)^{-1} (m\cdot\nabla P + M^2 J)
				+
				(m\cdot\nabla P + M^2 J) \nabla (-\Delta)^{-1} P
			\bigr) (1+s,y)
		\bigr\}
	dyds\\
	+& \rho_1 (t) + \cdots + \rho_7 (t).
\end{split}
\end{equation}
We can show that
\begin{equation}\label{th4bs1}
	\left\| \rho_1 (t) \right\|_{L^q (\mathbb{R}^2)}
	+
	\left\| \rho_2 (t) \right\|_{L^q (\mathbb{R}^2)}
	+
	\left\| \rho_3 (t) \right\|_{L^q (\mathbb{R}^2)}
	=
	o \bigl( t^{-2 (1-\frac1q) -2} \bigr)
\end{equation}
as $t\to\infty$ for $1 \le q \le \infty$ from the similar way as in the proof of Theorem \ref{th2}.
Indeed, we divide $\rho_1$ into $\rho_1 = \rho_{1,1} + \rho_{1,2}$, where
\[
\begin{split}
	\rho_{1,1} (t)
	=&
	\int_0^{t/2} \int_{\mathbb{R}^2}
		\biggl(
			\nabla P(t-s,x-y)
			-
			\sum_{|\beta| \le 1} \nabla^\beta \nabla P (t-s,x) (-y)^\beta
		\biggr)\\
		&\hspace{5mm}\cdot
		\bigl\{
		u\nabla(-\Delta)^{-1}u(s,y) -M^2 P\nabla(-\Delta)^{-1}P (s,y)\\
		&\hspace{10mm}-
		M \left( P \nabla(-\Delta)^{-1} (m\cdot\nabla P + M^2 J)
		+ (m\cdot\nabla P + M^2 J) \nabla(-\Delta)^{-1} P \right) (1+s,y)
	\bigr\} dyds,\\
	\rho_{1,2} (t)
	=&
	\sum_{|\beta| = 1} \int_0^{t/2} \int_{\mathbb{R}^2}
		\bigl(
			\nabla^\beta \nabla P(t-s,x)
			-
			\nabla^\beta \nabla P (t,x)
		\bigr)\\
		&\hspace{5mm}\cdot
		(-y)^\beta \bigl\{
		u\nabla(-\Delta)^{-1}u(s,y) -M^2 P\nabla(-\Delta)^{-1}P (s,y)\\
		&\hspace{10mm}-
		M \left( P \nabla(-\Delta)^{-1} (m\cdot\nabla P + M^2 J)
		+ (m\cdot\nabla P + M^2 J) \nabla(-\Delta)^{-1} P \right) (1+s,y)
	\bigr\} dyds.
\end{split}
\]
We consider only $\rho_{1,1}$, and split it as
\[
\begin{split}
	\rho_{1,1} (t)
	=&
	\int_0^{t/2} \int_{\mathbb{R}^2}
		\biggl( \nabla P (t-s,x-y) - \sum_{|\beta| \le 1} \nabla^\beta \nabla  P (t-s,x) (-y)^\beta \biggr)\\
		&\hspace{5mm} \cdot
		u (s) \nabla (-\Delta)^{-1} \left( u(s) - M P (s) - m\cdot\nabla P (1+s) - M^2 J (1+s) \right)
	dyds\\
	&+
	M\int_0^{t/2} \int_{\mathbb{R}^2}
		\biggl( \nabla P (t-s,x-y) - \sum_{|\beta| \le 1} \nabla^\beta \nabla  P (t-s,x) (-y)^\beta \biggr)\\
		&\hspace{5mm} \cdot
		\left( u(s) - M P(s) - m \cdot \nabla P (1+s) - M^2 J(1+s) \right) \nabla (-\Delta)^{-1} P (s)
	dyds\\
	&+
	\int_0^{t/2} \int_{\mathbb{R}^2}
		\biggl( \nabla P (t-s,x-y) - \sum_{|\beta| \le 1} \nabla^\beta \nabla  P (t-s,x) (-y)^\beta \biggr)\\
		&\hspace{5mm} \cdot
		\left( u(s) - M P (s) \right)
		\nabla (-\Delta)^{-1} \left( m\cdot\nabla P + M^2 J \right) (1+s)
	dyds.
\end{split}
\]
%
The similar procedure as in the proof of Theorem \ref{th2} with the aid of Propositions \ref{tl4} and \ref{HLSth4} leads that
\[
\begin{split}
	&\biggl\|
		\int_0^{t/2} \int_{\mathbb{R}^2}
			\biggl( \nabla P (t-s,x-y) - \sum_{|\beta| \le 1} \nabla^\beta \nabla  P (t-s,x) (-y)^\beta \biggr)\\
			&\hspace{5mm} \cdot
			u (s) \nabla (-\Delta)^{-1} \left( u(s) - M P (s) - m\cdot\nabla P (1+s) - M^2 J (1+s) \right)
		dyds
	\biggr\|_{L^q (\mathbb{R}^2)}\\
	+&
	\biggl\|
		\int_0^{t/2} \int_{\mathbb{R}^2}
			\biggl( \nabla P (t-s,x-y) - \sum_{|\beta| \le 1} \nabla^\beta \nabla  P (t-s,x) (-y)^\beta \biggr)\\
			&\hspace{5mm} \cdot
			\left( u(s) - M P(s) - m \cdot \nabla P (1+s) - M^2 J(1+s) \right) \nabla (-\Delta)^{-1} P (s)
		dyds
	\biggr\|_{L^q (\mathbb{R}^2)}
	=
	o \bigl( t^{-2(1-\frac1q) - 2} \bigr)
\end{split}
\]
as $t\to\infty$ for $1 \le q \le \infty$.
Taylor's theorem provides that
\[
\begin{split}
	&\int_0^{t/2} \int_{\mathbb{R}^2}
		\biggl( \nabla P (t-s,x-y) - \sum_{|\beta| \le 1} \nabla^\beta \nabla  P (t-s,x) (-y)^\beta \biggr)\\
		&\hspace{5mm} \cdot
		\left( u(s) - M P (s) \right)
		\nabla (-\Delta)^{-1} \left( m\cdot\nabla P + M^2 J \right) (1+s)
	dyds\\
	=&
	\sum_{|\beta| = 2} \int_0^{t/2} \int_{\mathbb{R}^2} \int_0^1
		\frac{\nabla^\beta \nabla P (t-s,x-y+\lambda y)}{\beta!}\\
		&\hspace{5mm}\cdot
		\lambda (-y)^\beta \left( u(s) - M P (s) \right)
		\nabla (-\Delta)^{-1} \left( m\cdot\nabla P + M^2 J \right) (1+s)
	d\lambda dyds.
\end{split}
\]
Thus, by Lemma \ref{decay-lin} and Proposition \ref{ap}, we have that
\[
\begin{split}
	&\biggl\| \int_0^{t/2} \int_{\mathbb{R}^2}
		\biggl( \nabla P (t-s,x-y) - \sum_{|\beta| \le 1} \nabla^\beta \nabla  P (t-s,x) (-y)^\beta \biggr)\\
		&\hspace{5mm} \cdot
		\left( u(s) - M P (s) \right)
		\nabla (-\Delta)^{-1} \left( m\cdot\nabla P + M^2 J \right) (1+s)
	dyds \biggr\|_{L^q (\mathbb{R}^2)}\\
	\le&
	C \int_0^{t/2}
		(t-s)^{-2(1-\frac1q) -3}
		\bigl\| u(s)-MP(s) \bigr\|_{L^1 (\mathbb{R}^2)} \bigl\| |y|^2 \nabla (-\Delta)^{-1} \left( m\cdot\nabla P + M^2 J \right) (1+s) \bigr\|_{L^\infty (\mathbb{R}^2)}
	ds\\
	\le&
	C \int_0^{t/2}
		(t-s)^{-2(1-\frac1q) -3} (1+s)^{-1} \log (e+s)
	ds
	=
	o \bigl( t^{-2(1-\frac1q) - 2} \bigr).
\end{split}
\]
Here we used the relation $\sup_{s > 0} \| |y|^2 \nabla (-\Delta)^{-1} \left( m\cdot\nabla P + M^2 J \right) (1+s) \|_{L^\infty (\mathbb{R}^2)} < \infty$.
Indeed, since
\[
	\nabla (-\Delta)^{-1} J(1)
	=
	\sum_{j=1}^2 \int_0^1 \nabla \partial_j (-\Delta)^{-1} P (1-s) * (P\partial_j (-\Delta)^{-1}P) (s) ds,
\]
we see from the H\"ormander-Mikhlin-type estimate that
\[
	\left| \nabla (-\Delta)^{-1} J(1) \right|
	\le
	C \left( 1+|y|^2 \right)^{-1}.
\]
A coupling of this and \eqref{scJ} yields that $\sup_{s>0} \| |y|^2 \nabla (-\Delta)^{-1} J (1+s) \|_{L^\infty (\mathbb{R}^2)} < \infty$.
Analogously we obtain that $\sup_{s>0} \| |y|^2 \nabla^2 (-\Delta)^{-1} P (1+s) \|_{L^\infty (\mathbb{R}^2)} < \infty$.
Similarly, we can treat $\rho_2$, and confirm that
\[
\begin{split}
	&\int_0^\infty \int_{\mathbb{R}^2}
		\bigl| y_j
		\bigl( u\nabla(-\Delta)^{-1}u -M^2 P\nabla(-\Delta)^{-1}P\\
	&-
	M \left( P \nabla(-\Delta)^{-1} (m\cdot\nabla P + M^2 J)
	+ (m\cdot\nabla P + M^2 J) \nabla(-\Delta)^{-1} P \right) \bigr|
	dyds
	< + \infty.
\end{split}
\]
Moreover we have the estimate for $\rho_3$.
Taylor's theorem yields that
\[
\begin{split}
	\rho_4 (t)
	=&
	\sum_{|\beta|=2} \int_0^{t/2} \int_{\mathbb{R}^2} \int_0^1 \int_0^1
		\frac{\nabla^\beta \nabla P(t-s,x-y+\lambda y)}{\beta!}
		\lambda (-y)^\beta\\
		&\hspace{5mm}\cdot
		\partial_t \left( P\nabla(-\Delta)^{-1}(m\cdot\nabla) P
		+ (m\cdot\nabla) P \nabla(-\Delta)^{-1}P \right) (s+\mu,y)
	d\mu d\lambda dyds
\end{split}
\]
and
\[
\begin{split}
	\rho_5 (t)
	=&
	\int_{t/2}^t \int_0^1
		P(t-s) *
		\nabla \cdot \partial_t \left( P\nabla(-\Delta)^{-1}(m\cdot\nabla) P
		+ (m\cdot\nabla) P \nabla(-\Delta)^{-1}P \right) (s+\mu,y)
	d\mu ds.
\end{split}
\]
For $1 \le q \le \infty$, we see from Lemma \ref{decay-lin} that
\begin{equation}\label{th4bs4}
\begin{split}
	&\left\| \rho_4 (t) \right\|_{L^p (\mathbb{R}^2)} + \left\| \rho_5 (t) \right\|_{L^q (\mathbb{R}^2)}\\
	\le&
	C \int_0^{t/2} \int_0^1
		(t-s)^{-2(1-\frac1q)-3}
		(s+\mu)^{-1}
	d\mu ds
	+
	C \int_{t/2}^t \int_0^1
		(s+\mu)^{-2(1-\frac1q)-4}
	d\mu ds\\
	=&
	o \bigl( t^{-2(1-\frac1q)-2} \bigr)
\end{split}
\end{equation}
as $t\to\infty$.
Analogously
\begin{equation}\label{th4bs6}
	\left\| \rho_6 (t) \right\|_{L^q (\mathbb{R}^2)}
	+
	\left\| \rho_7 (t) \right\|_{L^q (\mathbb{R}^2)}
	=
	o \bigl( t^{-2 (1-\frac1q) - 2} \bigr)
\end{equation}
as $t\to\infty$ for $1 \le q \le \infty$.
Applying \eqref{th4bs1}-\eqref{th4bs6} to \eqref{th4bs}, we complete the proof.
\hfill$\square$

\subsection{Properties of the correction terms}
Before closing this paper, we confirm the basic properties of the correction terms in the theorems.
\begin{proposition}\label{proptJ}
	The function $\tilde{J}$ in \eqref{tJ} satisfies \eqref{tJp}.
\end{proposition}
\begin{proof}
It suffices to show that the first term on $\tilde{J}$ is well-defined.
Since $\int_{\mathbb{R}^3} P\nabla(-\Delta)^{-1}P dy = 0$, we see from Taylor's theorem that
\[
\begin{split}
	&\int_0^{t/2} \int_{\mathbb{R}^3}
		\left( \nabla P (t-s,x-y) + (y\cdot\nabla) \nabla P (t,x) \right)
		\cdot \left( P \nabla (-\Delta)^{-1} P \right) (s,y)
	dyds\\
	=&
	\sum_{|\beta| = 2} \int_0^{t/2} \int_{\mathbb{R}^3} \int_0^1
		\frac{\nabla^\beta \nabla P(t-s,x-y+\lambda y)}{\beta!} \lambda
		\cdot (-y)^\beta (P\nabla(-\Delta)^{-1}P) (s,y)
	d\lambda dyds.
\end{split}
\]
Hence Lemma \ref{decay-lin} leads that
\[
\begin{split}
	&\biggl\| \int_0^{t/2} \int_{\mathbb{R}^3}
		\left( \nabla P (t-s,x-y) + (y\cdot\nabla) \nabla P (t,x) \right)
		\cdot \left( P \nabla (-\Delta)^{-1} P \right) (s,y)
	dyds \biggr\|_{L^p (\mathbb{R}^3)}\\
	\le&
	C \int_0^{t/2}
		(t-s)^{-3(1-\frac1p)-3}
		\left\| |y|^2 (P\nabla(-\Delta)^{-1}P)(s,y) \right\|_{L^1 (\mathbb{R}^3)}
	ds\\
	\le&
	C \int_0^{t/2}
		(t-s)^{-3(1-\frac1p)-3}
	ds
	\le
	C t^{-3(1-\frac1p)-2}
\end{split}
\]
for $1 \le p \le \infty$ and $t > 0$.
Thus $\tilde{J} \in C((0,\infty), L^1 (\mathbb{R}^3) \cap L^\infty (\mathbb{R}^3))$.
We see that $\lambda^5 \tilde{J} (\lambda t, \lambda x) = \tilde{J} (t,x)$ for any $\lambda > 0$.
Particularly $\tilde{J} (t,x) = t^{-5} \tilde{J} (1,t^{-1} x)$ and we obtain the second assertion.
\end{proof}
\begin{proposition}\label{propJ2}
	The function $J_2$ defined by \eqref{J2} satisfies \eqref{J2p}.
\end{proposition}
\begin{proof}
Since $\nabla (-\Delta)^{-1}$ is skew adjoint in $L^2 (\mathbb{R}^2)$, we see that
\[
	\int_{\mathbb{R}^2} \bigl(
			P \nabla (-\Delta)^{-1} (m\cdot\nabla P)
			+
			(m\cdot\nabla P) \nabla (-\Delta)^{-1} P
		\bigr) (s,y)
	dy
	= 0.
\]
Moreover,
\[
	\int_{\mathbb{R}^2} y_j \bigl(
			P \nabla (-\Delta)^{-1} (m\cdot\nabla P)
			+
			(m\cdot\nabla P) \nabla (-\Delta)^{-1} P
		\bigr) (s,y)
	dy
	= 0
\]
since $y_j (P \nabla (-\Delta)^{-1} (m\cdot\nabla P) + (m\cdot\nabla P) \nabla (-\Delta)^{-1} P ) (s,y)$ is an odd function in $y_1$ or $y_2$.
Hence Taylor's theorem says that
\[
\begin{split}
	&\int_0^{t/2}
		\nabla P (t-s)
		*
		\bigl(
			P \nabla (-\Delta)^{-1} (m\cdot\nabla P)
			+
			(m\cdot\nabla P) \nabla (-\Delta)^{-1} P
		\bigr) (s)
	ds\\
	=&
	\sum_{|\beta| = 2} \int_0^{t/2} \int_{\mathbb{R}^2} \int_0^1
		\frac{\nabla^\beta \nabla P(t-s,x-y+\lambda y)}{\beta!} \lambda\\
		&\hspace{5mm}\cdot (-y)^\beta \bigl(
			P \nabla (-\Delta)^{-1} (m\cdot\nabla P)
			+
			(m\cdot\nabla P) \nabla (-\Delta)^{-1} P
		\bigr) (s,y)
	d\lambda dyds.
\end{split}
\]
Thus we see from Lemma \ref{decay-lin} that
\[
\begin{split}
	&\biggl\| \int_0^{t/2} \nabla P (t-s)
		*
		\bigl(
			P \nabla (-\Delta)^{-1} (m\cdot\nabla P)
			+
			(m\cdot\nabla P) \nabla (-\Delta)^{-1} P
		\bigr) (s)
	ds \biggr\|_{L^p (\mathbb{R}^2)}\\
	\le&
	C \int_0^{t/2}
		(t-s)^{-2 (1-\frac1p) - 3}
		\left\| |y|^2 \bigl(
			P \nabla (-\Delta)^{-1} (m\cdot\nabla P)
			+
			(m\cdot\nabla P) \nabla (-\Delta)^{-1} P
		\bigr) (s)
		\right\|_{L^1 (\mathbb{R}^2)}
	ds\\
	\le&
	C \int_0^{t/2}
		(t-s)^{-2 (1-\frac1p) - 3}
	ds
	\le
	C t^{-2 (1-\frac1p) - 2}
\end{split}
\]
for $1 \le p \le \infty$.
In  a similar procedure as in the proof of Proposition \ref{proptJ}, we see that
\[
\begin{split}
	&\int_0^{t/2} \int_{\mathbb{R}^2}
		\left( \nabla P (t-s,x-y) + (y\cdot\nabla) \nabla P(t,x) \right)\\
		&\hspace{15mm}\cdot
		\bigl(
			P \nabla (-\Delta)^{-1} J + J \nabla (-\Delta)^{-1} P
		\bigr)
		(s,y)
	dyds\\
	=&
	\sum_{|\beta| = 2} \int_0^{t/2} \int_{\mathbb{R}^2} \int_0^1
		\frac{\nabla^\beta \nabla P (t-s,x-y+\lambda y)}{\beta!} \lambda\\
		&\hspace{5mm}\cdot
		(-y)^\beta \bigl(
			P \nabla (-\Delta)^{-1} J + J \nabla (-\Delta)^{-1} P
		\bigr) (s,y)
	d\lambda dyds
\end{split}
\]
and
\[
\begin{split}
	&\biggl\| \int_0^{t/2} \int_{\mathbb{R}^2}
		\left( \nabla P (t-s,x-y) + (y\cdot\nabla) \nabla P(t,x) \right)\\
		&\hspace{15mm}\cdot
		\bigl(
			P \nabla (-\Delta)^{-1} J + J \nabla (-\Delta)^{-1} P
		\bigr)
		(s,y)
	dyds \biggr\|_{L^p (\mathbb{R}^2)}
	\le
	C t^{-2 (1-\frac1p) - 2}.
\end{split}
\]
Therefore $J_2$ is well-defined in $C ((0,\infty),L^1 (\mathbb{R}^2) \cap L^\infty (\mathbb{R}^2))$.
The scaling-properties of $P$ say that $J_2 (t,x) = t^{-4} J_2 (1,t^{-1}x)$.
Hence we get the second assertion of \eqref{J2p}.
\end{proof}

\end{document}